\documentclass[12pt,twoside, reqno]{amsart}
\usepackage{amsmath}
\usepackage{amssymb}
\usepackage{graphicx}
\usepackage{mathrsfs}
\usepackage{eufrak}

\newtheorem{theorem}{Theorem}[section]
\newtheorem{lemma}{Lemma}[section]
\newtheorem{corollary}{Corollary}[section]

\theoremstyle{definition}
\newtheorem{definition}{Definition}[section]
\newtheorem{example}{Example}[section]
\newtheorem{remark}{Remark}[section]
\numberwithin{equation}{section}

\graphicspath{{./Figure/}}

\def\stretchx{\Bumpeq{\!\!\!\!\!\!\!\!{\longrightarrow}}}

\begin{document}

\title[Cutting Surfaces]
{Cutting surfaces and applications to periodic points and chaotic$-$like dynamics
{{ \footnote{\tiny{This work has been supported by the project
"Equazioni Differenziali Ordinarie e Applicazioni" (PRIN 2005).}}{$\;^1$}}}}
\author[M. Pireddu, F. Zanolin]
{Marina Pireddu and Fabio Zanolin}

\noindent
\begin{abstract}
\noindent In this paper we propose an elementary topological
approach which unifies and extends various different results
concerning fixed points and periodic points for maps defined on
sets homeomorphic to rectangles embedded in euclidean
spaces. We also investigate the associated discrete semidynamical
systems in view of detecting the presence of chaotic$-$like
dynamics.
\end{abstract}

\maketitle

\bigskip

\noindent {
{\small 2000 AMS  subject classification~: {47H10},
{54H20},  {37B10},  {37B99}, {54C30}. } }
\\
\noindent {
{\small keywords and phrases~:
Connected and locally arcwise connected metric spaces,
fixed points, periodic points,
Poincar\'{e}$-$Miranda theorem,
Leray$-$Schauder continuation theorem and its generalizations,
zeros of parameter dependent vector fields, chaotic dynamics,
topological horseshoes.
}}
\smallskip

\bigskip

\section{Introduction}\label{sec-1}
The celebrated Smale horseshoe
can be rightly considered as a prototypical
example in the study of complex systems.
It deals with a homeomorphism $\psi,$ defined on a set diffeomorphic to a rectangle
in a two$-$dimensional manifold, possessing an invariant set
$\Lambda$ such that $\psi|_{\Lambda}$ is conjugate to the two$-$sided
Bernoulli shift on two symbols. Such striking model exploits a simple
and elegant geometric description in order to display the
main features associated to all
the various different definitions of chaos.
Since the beginning \cite{Sm-65, Sm-67}, its clear and intuitive geometrical
structure turned out to be very useful in the study of dynamical systems,
allowing a rigorous proof of the presence of a complex behavior
in several significant ODE models \cite[Ch.3]{Mo-73}.
One of the crucial assumptions regarding the
implementation of this method concerns the verification of
hyperbolicity. Usually, this concept
is meant as the existence of a splitting of the domain
into a part along which $\psi$ is contracting and another one where
$\psi$ is expanding, the two parts being filled
by graphs of lipschitzean functions (see \cite{Mo-73,Wi-88}).
However, in certain applications to differential systems, the verification of
hyperbolicity requires the smoothness of the involved maps and conditions on their jacobian matrix
which lead sometimes to formidable and difficult computations (see \cite[p.62]{Mo-73}).
This remark or similar ones carried various authors to look for a class of relaxed
assumptions in view of producing a structure as rich as before,
but possibly replacing the hyperbolicity hypothesis with topological conditions
which, in some cases, require the knowledge of the behavior of $\psi$ only on some
subsets of its domain (for instance, at the boundary of certain sets).
Results in this direction
were obtained by Easton \cite{Ea-75}, Burns and Weiss \cite{BuWe-95},  Mischaikow and Mrozek
\cite{MiMr-95}, Szymczak \cite{Sz-96},  Zgliczy\'nski \cite{Zg-96},  Srzednicki and W{\'o}jcik \cite{SrWo-97}
and further developed in subsequent works (see, for instance,
\cite{MrWo-05, PiZa-05, Sr-00, Zg-99, Zg-01, ZgGi-04}
and the references therein,
just to quote a few samples from a broad bibliography).
In these papers usually the authors prove the existence of a compact set
$\Lambda$ which is (forward) invariant (either for a given map $\psi$ or for an iterate of it)
and the existence of a continuous surjection $g: \Lambda \to \Sigma_{m}:= \{0,\dots,m-1\}^{\mathbb Z}$
(or $g: \Lambda \to \Sigma_{m}^+:= \{0,\dots,m-1\}^{\mathbb N}$)
which provides a semiconjugation of $\psi|_{\Lambda}$ with the
two$-$sided (respectively one$-$sided) Bernoulli shift.
In many cases the authors also show that the inverse image through $g$ of a periodic sequence
of $m$ symbols contains a periodic point of $\psi$ in $\Lambda.$
The tools employed in the related proofs are based on different topological methods, like
the Conley index theory (with associated homological or cohomological invariants)
or some more or less sophisticated fixed point methods (as the topological degree,
the fixed point index or the Lefschetz theory) which usually require
the verification of suitable conditions (for a flow or for a map) at the boundary
of a certain domain containing the invariant set in its interior.

A different generalization of the Smale's horseshoe has been obtained
in recent years
by Kennedy and Yorke \cite{KeYo-01} and Kennedy, Ko\c{c}ak and Yorke
\cite{KeKoYo-01}, who developed the theory of the so-called
\textit{topological horseshoes} in the frame of metric spaces.
This theory concerns the study of the behavior of a continuous map $f: X\supseteq Q \to X$ and of its iterates,
where $Q$ is a compact subset of a metric space $X.$
Two disjoint compact subsets $\mbox{\it{end}}_0$ and $\mbox{\it{end}}_1$ of $Q$ are selected
and the crossing number $m$ is defined as the largest number of disjoint preconnections
contained in any connection. A connection $\Gamma\subseteq Q$ is a compact connected set
(a continuum) with $\Gamma\cap \mbox{\it{end}\,}_0\not=\emptyset$
and $\Gamma\cap \mbox{\it{end}\,}_1\not=\emptyset,$ while a preconnection $\gamma\subseteq\Gamma$
is a compact set such that $f(\gamma)$ is a connection.
Under a few more technical conditions
(see \cite{KeKoYo-01, KeYo-01} for the precise statements)
and if $m\geq 2,$ the authors
prove the existence of a compact invariant set
$Q_I\subseteq Q$ such that $f|_{Q_I}$ is semiconjugate
(via a surjective map $g$)
to a Bernoulli shift on $m$ symbols.
Applications  of such results
to different ODE models have been shown in \cite{HuYa-05, YaLi-06, YaLiHu-05, YaYa-04}
and in recent related papers. The great generality of the setting in
\cite{KeKoYo-01, KeYo-01} does not lead to the conclusion that the inverse image
(through the surjection $g$)
of a periodic sequence in $\Sigma_m$ must necessarily contain
a periodic point in $Q_I\,.$ In fact, a specific example of an invariant set
$Q_I$ without periodic points is presented in \cite[p.2520]{KeYo-01}.

Another approach has been proposed by Papini and the second author in
\cite{PaZa-02}, motivated by the study of the Poincar\'{e} map associated to a
second order nonlinear differential equation with sign changing weight
\cite{PaZa-00}. In \cite{PaZa-02} and some subsequent
more general works \cite{PaZa-04a, PaZa-04b},
the authors considered continuous planar maps which possess the property of
stretching the paths joining two opposite sides of a topological rectangle.
In this case the paths play a role analogous to that of the connections in Kennedy and Yorke theory.
Such a special configuration permits to obtain not only chaotic dynamics, but
also the existence of fixed points and periodic points
(by means of elementary topological arguments). In this manner one could complement some
results in the two$-$dimensional setting (like those in \cite{YaYa-04}) and provide
the existence of infinitely many periodic points (of any order) as well.
Other applications of the results in \cite{PaZa-02} can be found in \cite{DaPa-04, PaZa-04a},
dealing with second order ODEs.
See also \cite{PaZa->} for more information,
although we have to warn that a precise relationship between
these works and the theory of topological horseshoes
as presented in \cite{KeKoYo-01, KeYo-01} was not explicitly stated therein.
On the contrary, such relationship will be described in the present work (see
Lemma \ref{lem-3c.1}), including some details elsewhere missing.

\medskip
The aim of this paper is twofold. In fact we provide an elementary topological tool that, on the one hand,
allows us to extend to the higher dimension
the theory of \cite{PaZa-02,PaZa-04a,PaZa-04b} and to make a comparison with
the above recalled results on chaotic$-$like dynamics, while,
on the other hand, enables to
generalize and unify some previous theorems about the existence of fixed points and
periodic points for continuous maps in euclidean spaces. In particular,
we give a simplified proof of a recent result by Kampen \cite{Ka-00}
as well as we present a generalization of some preceding theorems
about periodic points associated to Markov partitions \cite{Zg-01}.
As remarked above, one of the purposes of our work is also that of
obtaining a sufficiently rich structure (suitable, for instance,
to guarantee the existence of infinitely many periodic points),
without the need of too sophisticated methods.
In fact, our main tools are, respectively, a modified version of the
classical Hurewicz$-$Wallman intersection lemma \cite[D), p.40]{HuWa-41}, \cite[p.72]{En-78}
and the fundamental theorem of Leray$-$Schauder \cite[Th\'{e}or\`{e}me Fondamental]{LeSh-34}.
The former result (also referred to Eilenberg$-$Otto \cite[Th.3]{EiOt-38},
according to \cite[Theorem on partitions, p.100]{DuGr-03}) is one of the basic lemmas
in dimension theory and it is known to be one of the equivalent versions of the
Brouwer fixed point theorem. It may be interesting to observe that extensions of
the intersection lemma led to generalizations of the Brouwer fixed point theorem to
some classes of possibly noncontinuous functions \cite{Wh-67}.
The other result we use concerns
topological degree theory and has found important applications in
nonlinear analysis and differential equations (see \cite{Ma-97}).
Actually, in Theorem \ref{th-3b.1} of Section \ref{sec-3b}, we'll
employ a more general version of the Leray$-$Schauder continuation theorem
due to Fitzpatrick, Massab\'{o} and Pejsachowicz \cite{FiMaPe-86}.
However, we point out that for our applications to the study of periodic
points and chaotic$-$like dynamics we rely on the classical
Leray$-$Schauder fundamental theorem.
Another particular feature of our approach consists in a combination of the Poincar\'{e}$-$Miranda
theorem \cite{Ku-97,Ma-00} with the properties of topological surfaces cutting the
arcs between two given sets.
The corresponding details are widely described in
Section \ref{sec-2} and Section \ref{sec-3a}.
In regard to the Poincar\'{e}$-$Miranda theorem, which is an $N-$dimensional version of the
intermediate value theorem (see Theorem \ref{th-2b.1}), we would like to
recall also the recent work \cite{BaCsGa->}, where the authors show its effectiveness
in detecting chaotic dynamics in planar dynamical systems.

\medskip
This article is organized as follows. Section \ref{sec-2} is
devoted to the presentation of some topological lemmas concerning
the relationship between cutting surfaces and zero sets of
continuous real valued functions. Analogous results can be found, often in a more implicit form, in different contexts. Since for our
applications we need a specific version of the statements, we
give an independent proof with all the details. The remainder of the
paper is divided in three parts. In Section \ref{sec-3a} we present a
variant of Hurewicz$-$Wallman theorem (Theorem \ref{th-2b.2}), that (likewise \cite{EiOt-38,HuWa-41}) guarantees a nonempty
intersection among $N$ closed topological surfaces which are in good
position with respect to the faces of an $N-$dimensional hypercube.
Such result is then applied in order to provide an extension of
some recent theorems about fixed points and periodic points for
continuous mappings in euclidean spaces. In particular, we
generalize (with a simplified proof) a theorem by Kampen
\cite{Ka-00} (see Corollary \ref{cor-2b.1}) and obtain also an
extension of a result by Zgliczy\'nski \cite{Zg-01} (see Theorem
\ref{th-2b.z}). In Section \ref{sec-3b} we consider the case in
which, roughly speaking, the number of intersecting surfaces is
smaller than the dimension of the space. After presenting a
general situation in Theorem
\ref{th-3b.2a}, we restrict ourselves to the case of the
intersection of $N-1$ continua embedded in a topological space
homeomorphic to the $N-$dimensional hypercube (see Corollary
\ref{cor-3b.1}). This latter result turns out to be our main
ingredient in Section \ref{sec-3c} for the study of the dynamics of continuous maps which
possess, in a very broad sense, a one$-$dimensional expansive
direction. Indeed, it allows to prove a fixed point theorem (see
Theorem \ref{th-3c.1}) for maps defined on topological rectangles.
The main hypothesis for our fixed point theorem requires that the map $\psi$
stretches (across $X$) the paths joining two disjoint subsets $X_{\ell}$ and $X_{r}$
of the boundary of the topological rectangle $X.$
Formally, such stretching condition is expressed as follows:
there exists a compact set ${\mathcal K}\subseteq X$ on which $\psi$ is
continuous and such that, for any path $\gamma$ connecting in $X$
the two sides $X_{\ell}$ and $X_{r}\,,$ there exists a sub-path $\sigma$ of
$\gamma$ contained in ${\mathcal K},$ whose image through $\psi$ is
contained in $X$ and joins the same two sides of the boundary (as shown in Figure 1).
A special feature of our theorem is that it ensures the existence
of a fixed point for $\psi$ in ${\mathcal K},$ that, in turns, yields to the existence
of multiple fixed points when the stretching condition is satisfied
with respect to a certain number of pairwise disjoint subsets ${\mathcal K}_i$'s
of $X.$ An application of the theorem to the iterates of $\psi$ will
then allow to obtain a set of infinitely many periodic points with a
complex structure. More generally, we define, for a mapping $\psi,$
a concept of stretching between two (possibly different) oriented $N$-dimensional
rectangles $(X,X^-)$ and $(Y,Y^-),$ where the $[\cdot]^-$-sets are the union of two
suitably chosen left and right sides of the boundary (see Definition \ref{def-3c.1} and
Figure 1).

{
\includegraphics[scale=0.45]{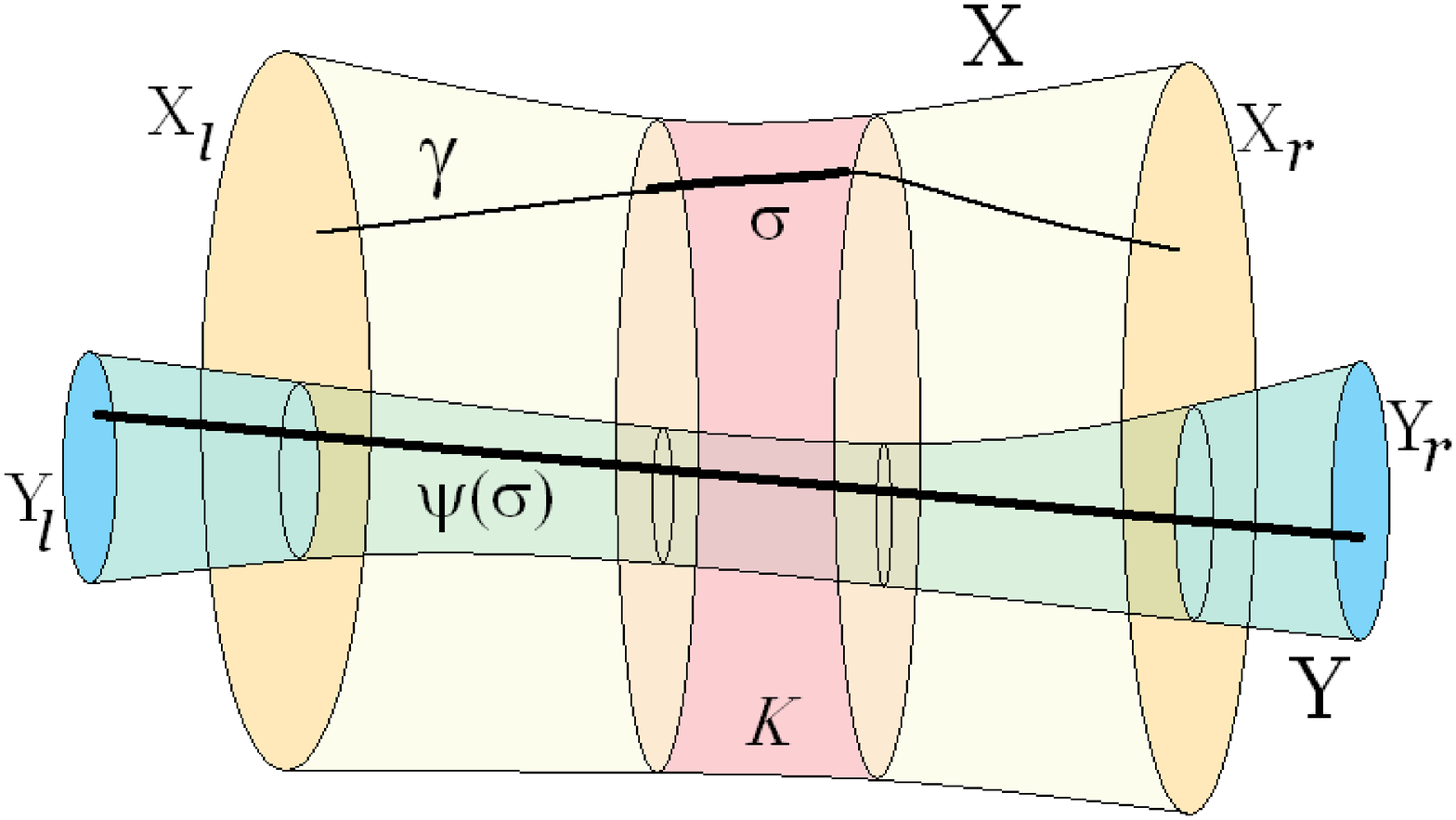}
}
\begin{quote}
{\small{Figure 1. The tubular sets $X$ and $Y$ in the picture
represent two generalized 3$-$dimensional rectangles, in which we
have put in evidence the compact set $\mathcal K$ and the boundary
sets $X_{\ell}$ and $X_{r}$ as well as $Y_{\ell}$ and $Y_{r}$ (see
Definition \ref{def-3c.1} for more details). In this particular
case the map $\psi$ stretches the paths of $X$ not only across $Y,$
but also across $X$ itself and therefore the existence of a fixed
point for $\psi$ inside $\mathcal K$ is ensured by Theorem
\ref{th-3c.1}. Note that, differently from the classical Rothe and Brouwer theorems, we don't require that $\psi(\partial X)\subseteq X.$}}
\end{quote}

In this manner we fully extend to an arbitrary dimension the above recalled
results by Papini and Zanolin \cite{PaZa-02,PaZa-04a,PaZa-04b}.
As a consequence, also all the applications
to ordinary differential equations and dynamical systems contained in those
articles and in related ones
(like \cite{CaDaPa-02,DaPa-04,PaZa-00}) are, in principle, extendable to the higher dimension.
The specific applications, which would require a detailed treatment,
will be presented in a future work.

\medskip

We end this introductory section with a list of basic definitions
and notation. Some further concepts will be introduced later on in
the paper when needed.
\\
As usual we denote by ${\mathbb R},$ ${{\mathbb R}}^+:= [0,+\infty)\,$ and
${{\mathbb R}}^+_{0}:= \,]0,+\infty)\,$ the sets of reals, as well as the
nonnegative and positive real numbers, respectively. The sets of
integers ${\mathbb Z}$ and natural numbers ${\mathbb N} =
\{0,1,2,\dots\}$ are considered as well. For a subset $M$ of a
topological space $Z,$ we denote by $M^{^{^{\!\!\!\!\!{o}}}}$ and
$\overline{M}$ the interior and the closure of $M,$ respectively.
For a metric space $(X,d)$ we indicate with $B(x_0,r)$ the open
ball of center $x_0\in X$ and radius $r > 0.$ Similarly, given
$M\subseteq X$ ($M\not=\emptyset$), we define $B(M,r):= \{x\in X:
\exists w\in M,\; d(x,w) < r\}.$ By a continuum we mean a compact
and connected subset of a metric space (i.e. a metric continuum).

Let $Z$ be a topological space. By a \textit{path} $\gamma$ in $Z$ we mean a
continuous mapping (parameterized curve) $\gamma : {\mathbb R}\supseteq [a,b]\to Z$ and we denote by
${\overline{\gamma}}$ its range, that is
$${\overline{\gamma}}:=\gamma([a,b]).$$
A sub-path $\sigma$ of $\gamma$ is defined as
$$\sigma:= \gamma|_{[c,d]}\,,\quad \mbox{for } \; [c,d]\subseteq [a,b],$$
that is the restriction of $\gamma$ to a closed subinterval of its
domain. If $Z, Y$ are topological spaces and $\psi:Z\supseteq
D_{\psi}\to Y$ is a map which is continuous on a set ${\mathcal
M}\subseteq D_{\psi}\,,$ then for any path $\gamma$ in $Z$ with
$\overline{\gamma}\subseteq {\mathcal M},$ it follows that
$\psi\circ\gamma$ is a path in $Y$ with range equal to
$\psi({\overline{\gamma}}).$ As usual in the theory of (continuous)
parameterized curves, there is no loss of generality in assuming the
paths to be defined on $[0,1].$ In fact, if $\theta_1: [a_1,b_1]\to
Z$ and $\theta_2:[a_2,b_2]\to Z,$ with $a_i<b_i,\, i=1,2,$ are two
paths in $Z,$ we set $\theta_1\sim \theta_2$ if there exists a
homeomorphism $h$ of $[a_1,b_1]$ onto $[a_2,b_2]$ (i.e., a change
of variable in the parameter) such that $\theta_2(h(t)) =
\theta_1(t),\,$ $\forall\, t\in [a_1,b_1].$ It is easy to check that `` $\sim$
'' is an equivalence relation and that if $\theta_1\sim\theta_2\,,$
then the ranges of $\theta_1$ and $\theta_2$
coincide. Hence, for any path $\gamma$
we can find an equivalent path defined on $[0,1].$
\\
If $\gamma_1,\,\gamma_2:[0,1]\to Z$ are two paths in $Z$ with
$\gamma_1(1)=\gamma_2(0),$  we define the \textit{gluing of
$\gamma_1$ with $\gamma_2$}  as the path
$\gamma_1\star\gamma_2:[0,1]\to Z$ such that
\begin{displaymath}
\gamma_1\star\gamma_2(t):=\left\{ \begin{array}{ll}
\gamma_1(2t) & \textrm{for $0\le t\le \frac 1 2 ,$}\\
\gamma_2(2t-1) & \textrm{for $\frac 1 2\le t\le 1.$ }
\end{array} \right.
\end{displaymath}
Moreover, given a path $\gamma:[0,1]\to Z,$ we denote by
$\gamma^-:[0,1]\to Z$ the path having $\overline{\gamma}$ as support,
but run with reverse orientation, i.e. $\gamma^-(t):=\gamma(1-t),$
for all $t\in [0,1].$
\\
At last we recall a known definition. Let $Z$ be a topological
space. We say that $Z$ is {\itshape arcwise connected} if, given
any two different points $p,q\in Z,$ there is a path $\gamma:
[0,1]\to Z$ such that $\gamma(0) = p$ and $\gamma(1) = q.$ In the
case of a Hausdorff topological space $Z,$ the range
${\overline{\gamma}}$ of $\gamma$ turns out to be a locally connected
metric continuum (a Peano space according to \cite{HoYo-61}).
Then, if $Z$ is a metric space, the above definition of arcwise connectedness is equivalent
to the fact that, given any two points $p,q\in Z$ with $p\not= q,$
there exists an {\itshape arc} (that is the homeomorphic image of
a compact interval) contained in $Z$ and having $p$ and $q$ as
extreme points (see, e.g., \cite[pp.115--131]{HoYo-61}).

\section{Topological lemmas}\label{sec-2}

As already mentioned in the Introduction, in this section we present some topological lemmas in order to show the relationship between particular surfaces and zero sets of continuous real valued functions.
First of all, we need a definition.
\begin{definition}\label{def-2.1}
Let $X$ be an arcwise connected metric space. Let $A, B,
C\subseteq X$ be closed and nonempty sets with $A\cap
B=\emptyset.$ We say that \textit{$C$ cuts the arcs between $A$
and $B$} if for any path $\gamma: [0,1]\to X,$
with ${\overline{\gamma}}\cap A\neq\emptyset$ and ${\overline{\gamma}}\cap
B\neq\emptyset,$ it follows that ${\overline{\gamma}}\cap
C\neq\emptyset.$ In the sequel, if $X$ is a subspace of a larger
metric space $Z$ and we wish to stress the fact that we consider
only paths contained in $X,$ we make more precise our definition
by saying that \textit{$C$ cuts the arcs between $A$ and $B$ in $X$ }.
\end{definition}

\noindent
Such definition is a modification of the classical one regarding the cutting
of a space between two points in \cite{Ku-68}.
See also \cite{BeDiPe-02} for a more general concept in which
the authors consider a set $C$ which
intersects every connected set meeting two nonempty sets $A$ and $B.$
In the case in which $A$ and $B$ are the opposite faces of an $N-$dimensional cube,
J. Kampen \cite[p.512]{Ka-00} says that $C$ separates $A$ and $B.$ We prefer to use
the ``cutting'' terminology in order to avoid misunderstanding with other definitions
of separation which are more common in Topology.
In particular (see \cite{FeLeSh-99}), we remark that
our definition agrees with the usual one of cut
when $A,B,C$ are pairwise disjoint.

In the sequel, even when not explicitly mentioned, we assume that
the basic space $X$ is arcwise connected. In some of the next results
the local arcwise connectedness will be required too. With this respect, we
recall that
any connected and locally arcwise connected metric space is arcwise connected (see \cite[Th.2, p.253]{Ku-68}).

\begin{lemma}\label{lem-ABC.1}
Let $X$ be a connected and locally arcwise connected metric space and let
$A, B, C\subseteq X$ be closed and nonempty sets with $A\cap
B=\emptyset.$ Then $C$ cuts the arcs between $A$ and $B$ if and
only if there exists a continuous function $f: X\to {\mathbb R}$ such that
\begin{equation}\label{eq-AB.1}
f(x)\leq 0,\, \forall\, x\in A,\qquad f(x) \geq 0,\, \forall\, x\in B
\end{equation}
and
\begin{equation}\label{eq-C.1}
C = \{x\in X: f(x) = 0\}.
\end{equation}
\end{lemma}
\begin{proof}
Assume there exists a continuous function $f: X\to {\mathbb R}$ satisfying
\eqref{eq-AB.1} and \eqref{eq-C.1}. Let $\gamma: [0,1]\to X$ be a continuous path
such that $\gamma(0)\in A$ and $\gamma(1)\in B.$ We want to prove that
${\overline{\gamma}}\cap C\neq\emptyset.$ Indeed, for
the composite continuous function
$\theta:=f\circ\gamma: [0,1]\to {\mathbb R},$ we have that $\theta(0)\leq 0 \leq \theta(1)$ and therefore
Bolzano theorem ensures the existence of $t^*\in [0,1]$ with $\theta(t^*)=0.$ This means that
$\gamma(t^*)\in C$ and therefore ${\overline{\gamma}}\cap C\neq\emptyset.$
Thus we have proved that $C$ cuts the arcs between $A$ and $B.$

Conversely, let us assume that $C$ cuts the arcs between $A$ and $B.$
We define the following auxiliary functions
\begin{equation*}
\rho: X\to {{\mathbb R}}^+\,,
\end{equation*}
\begin{equation}\label{eq-f.1}
\rho(x):= \mbox{dist}(x,C),\;\forall\, x\in X
\end{equation}
and
\begin{equation*}
\mu: X\to \{-1,0,1\},
\end{equation*}
\begin{equation}\label{eq-f.2}
\mu(x):=
\quad
\left
\{
\begin{array}{ll}
~\,0&\; \mbox{ if } \, x\in C,\\
\!-1 &\; \mbox{ if } \, x\not\in C \;
\mbox{and there exists a path $\gamma_x:[0,1]\to X\setminus C$} \\
~\quad &\mbox{
\,such that $\gamma_x(0)\in A$ and $\gamma_x(1) = x,$}\\
~\,1&\; \mbox{ elsewhere.}\\
\end{array}
\right.
\end{equation}
Observe that $\rho$ is a continuous function with $\rho(x) = 0$ if and only if $x\in C$
and also $\mu(x) = 0$ if and only if $x\in C.$ Moreover,  $\mu$ is bounded.

Let $x_0\not\in C.$ We claim that $\mu$ is continuous in $x_0\,.$ Actually, $\mu$
is locally constant on $X\setminus C.$ Indeed, since $x_0\in X\setminus C$ (an open set) and $X$
is locally arcwise connected, there is a neighborhood $U_{x_0}$ of $x_0$ with
$U_{x_0}\subseteq X\setminus C,$ such that for each $x\in U_{x_0}$ there exists a path $\sigma_{x_0,x}$
joining $x_0$ to $x$ in $U_{x_0}\,.$ Clearly, if there is a path $\gamma_{a,x_0}$ in $X\setminus C$
joining some point $a\in A$ with $x_0\,,$ then the path $\gamma_{a,x_0}\star\sigma_{x_0,x}$
connects $a$ to $x$ in $X\setminus C.$ This proves that if $\mu(x_0) = -1,$ then $\mu(x) = -1$
for every $x\in U_{x_0}\,.$ On the other hand, if there is a path $\gamma_{a,x}$ in
$X\setminus C$
which connects some point $a\in A$ to $x\in U_{x_0}\,,$ then, the path
$\gamma_{a,x}\star\sigma^-_{x_0,x}$ connects $a$ to $x_0\,$ in $X\setminus C.$
This shows that if $\mu(x_0) = 1$ (that is, it is not possible to connect $x_0$ to any point
of $A$ in $X\setminus C$ using a path), then $\mu(x) = 1$ for every $x\in U_{x_0}$
(that is, it is not possible to connect any point $x\in U_{x_0}$ to any point
of $A$ in $X\setminus C$ using a path).
We can now define
\begin{equation}\label{eq-f.3}
f: X\to {\mathbb R}, \qquad f(x):= \rho(x) \mu(x).
\end{equation}
Clearly, $f(x)= 0$ if and only if $x\in C$ and, moreover, $f$ is continuous.
Indeed, if $x_0\not \in C,$ we have that $f$ is continuous at $x_0$ because both $\rho$ and $\mu$ are
continuous in $x_0\,.$ If $x_0\in C$ and $x_n\to x_0$ (as $n\to \infty$), then $\rho(x_n)\to 0$
and $|\mu(x_n)|\leq 1,$ so that $f(x_n)\to 0 = f(x_0).$
Finally, by the definition of $\mu$ in \eqref{eq-f.2}, it holds that $\mu(a) = -1,$ for every $a\in A\setminus C$ and therefore, for such an $a$, it holds that $f(a) < 0.$ On the contrary, if we suppose that $b\in B\setminus C,$ we must have $\mu(b) =1.$ In fact, by the cutting condition,
there is no path connecting in $X\setminus C$ the point $b$ to any point of $A.$ Therefore, in this case we have $f(b) > 0.$
The proof is complete.
\end{proof}

\noindent Considering the functions $\mu$ and $f$ as in
\eqref{eq-f.2}, \eqref{eq-f.3}, we see that our definition,
although adequate from the point of view of the proof of Lemma
\ref{lem-ABC.1}, perhaps does not represent
an optimal choice. For instance, we would like the sign
of $f$ to coincide for all the points located ``at the same side''
of $A$ (respectively of $B$) with respect to $C.$ Having this
request in mind, we propose a different definition for the
function $\mu$ (see \eqref{eq-f.2b}). First of all, we introduce
the following sets that we call \textit{the side of $A$ in $X$}
and \textit{the side of $B$ in $X$}, respectively.
\begin{equation*}
{\mathcal S}(A):=\{x\in X: \, {\overline{\gamma}}\cap A\not=\emptyset,\,\forall\, \mbox{path }\,\gamma: [0,1]\to X,
\mbox{with } \gamma(0) = x, \gamma(1)\in B\},
\end{equation*}
\begin{equation*}
{\mathcal S}(B):=\{x\in X: \, {\overline{\gamma}}\cap B\not=\emptyset,\,\forall\, \mbox{path }\,\gamma: [0,1]\to X,
\mbox{with } \gamma(0) = x, \gamma(1)\in A\},
\end{equation*}
that is, a point $x$ belongs to the side of $A$ (resp. to the side of $B$)
if whenever we try to connect $x$ to $B$ (resp. to $A$) by a path,
we first meet $A$ (resp., we first meet $B$). By definition, it follows that
\begin{equation*}
A\subseteq {\mathcal S}(A),\quad B\subseteq {\mathcal S}(B).
\end{equation*}

\begin{lemma}\label{lem-ABC.3}
Let $X$ be a connected and locally arcwise connected metric space and let
$A, B\subseteq X$ be closed and nonempty sets with $A\cap
B=\emptyset.$ Then ${\mathcal S}(A)$ and ${\mathcal S}(B)$ are closed and, moreover,
\begin{equation*}
{\mathcal S}(A)\cap {\mathcal S}(B) = \emptyset.
\end{equation*}
\end{lemma}
\begin{proof}
First of all, we prove that ${\mathcal S}(A)$ is closed by
checking that if $w\not\in {\mathcal S}(A)$ then there is a
neighborhood $U_w$ of $w$ with $U_w\subseteq X\setminus {\mathcal
S}(A).$ Indeed, if $w\not\in {\mathcal S}(A),$ there exists a path
$\gamma:[0,1]\to X$ with $\gamma(0) = w$ and $\gamma(1)=b\in B$
and such that $\gamma(t)\not\in A,$ for every $t\in [0,1].$ Since
$A$ is a closed set and $X$ is locally arcwise connected, there
exists an arcwise connected open set $V_w$ with $w\in V_w\subseteq
X\setminus A.$ Hence, for every $x\in V_w,$ there is a path
$\sigma_x$ connecting $x$ to $w$ in $V_w\,.$ As a consequence, we
find that the path $\gamma_x:=\sigma_x\star\gamma$ connects $x\in
V_w$ to $b\in B$ with $\overline{\gamma_x}\subseteq X\setminus A.$
Clearly, the open neighborhood $U_w:= V_w$ satisfies our
requirement and this proves that $X\setminus {\mathcal S}(A)$ is
open. The same argument ensures that ${\mathcal S}(B)$ is closed.

It remains to show that ${\mathcal S}(A)$ and ${\mathcal S}(B)$
are disjoint. Since $A\cap B =\emptyset,$
it follows immediately from the definition that
$$A\cap {\mathcal S}(B)= \emptyset,\quad
B\cap {\mathcal S}(A)= \emptyset.$$
Now assume, by contradiction, that there exists
$$x\in {\mathcal S}(A) \cap {\mathcal S}(B),$$
with $x\not\in A\cup B.$
Let $\gamma: [0,1]\to X$ be a path such that $\gamma(0) = x$
and $\gamma(1) = b\in B$ (we know that there exists a path like that
because $X$ is arcwise connected). The fact that $x\in {\mathcal S}(A)\setminus A$
implies that there is $t_1\in \, ]0,1[\,$ such that $\gamma(t_1) = a_1\in A.$
On the other hand, since
$x\in {\mathcal S}(B)\setminus B,$
there exists $s_1\in \, ]0,t_1[\,$ such that $\gamma(s_1) = b_1\in B.$
Proceeding by induction and using repeatedly the definition of ${\mathcal S}(A)$
and ${\mathcal S}(B)$ we obtain a sequence
$$t_1 > s_1 > t_2 > \dots > t_{j} > s_{j} > t_{j+1} > \dots > 0$$
with $\gamma(t_i) = a_i \in A$ and $\gamma(s_i) = b_i\in B.$
For $t^*= \inf t_n = \inf s_n \in [0,1[\,,$ we have that
$\gamma(t^*) = \lim a_n = \lim b_n \in A\cap B,$ a contradiction.
The proof is complete.
\end{proof}

\begin{lemma}\label{lem-ABC.4}
Let $X$ be a connected and locally arcwise connected metric space and let
$A, B, C\subseteq X$ be closed and nonempty sets with $A\cap
B=\emptyset.$ Then $C$ cuts the arcs between $A$ and $B$ if and
only if $C$ cuts the arcs between ${\mathcal S}(A)$ and ${\mathcal S}(B).$
\end{lemma}
\begin{proof}
One direction of the inference is obvious. In fact, every path joining $A$ with $B$
is also a path joining ${\mathcal S}(A)$ with ${\mathcal S}(B).$ Thus, if
$C$ cuts the arcs between ${\mathcal S}(A)$ and ${\mathcal S}(B),$
then it also cuts the arcs between $A$ and $B.$
Conversely, let us assume that $C$ cuts the arcs between $A$ and $B.$
We want to prove that $C$ cuts the arcs between ${\mathcal S}(A)$ and ${\mathcal S}(B).$
By the definition of ${\mathcal S}(A)$ and ${\mathcal S}(B),$ it is
straightforward to check that
$C$ cuts the arcs between $A$ and ${\mathcal S}(B)$ as well as
it cuts the arcs between ${\mathcal S}(A)$ and $B.$
Suppose now, by contradiction, that there exists a path
$$\gamma: [\tfrac 1 2,1] \to X\setminus C$$
such that
$\gamma(\tfrac 1 2) = a\in {\mathcal S}(A)\setminus A$ and
$\gamma(1) = b\in {\mathcal S}(B)\setminus B.$
We choose a point $a_0 \in A$ and connect it to $a\in {\mathcal S}(A)$ by a path
$\sigma: [0,\tfrac 1 2] \to X$ with $\sigma(0) = a_0$ and $\sigma(\tfrac 1 2) = a.$
We define also the new path $\gamma_0:= \sigma\star\gamma : [0,1]\to X,$
with $\gamma_0(0) = a_0\in A$ and $\gamma_0(1) = b\in {\mathcal S}(B).$
By the definition of ${\mathcal S}(B)$ we know that there exists $s_1\in \, ]0,1[\,$
such that $\gamma_0(s_1) = b_1\in B.$ Note that $b_1\not\in {\mathcal S}(A)$
(recall that $B\cap {\mathcal S}(A) = \emptyset$) and also $0 < s_1 < \tfrac 1 2$
(in fact if $\tfrac 1 2 < s_1 < 1,$ then, recalling that $\gamma_0(\tfrac 1 2) =
\gamma(\tfrac 1 2) = a \in {\mathcal S}(A)$ and $\gamma_0(s_1) =
\gamma(s_1) = b_1 \in B,$ there must be a ${\tilde{t}}\in [\tfrac 1 2, s_1]$
such that $\gamma({\tilde{t}}) \in C,$ a contradiction to the assumption on $\gamma$).
The restriction of the path
$\gamma_0$ to the interval $[s_1, \tfrac 1 2],$
defines a path joining $b_1\in B$ to $a\in {\mathcal S}(A).$ Therefore there exists
$t_1\in \,]s_1,\tfrac 1 2[\,$ such that
$\gamma_0(t_1) = a_1\in A.$ The restriction of the path
$\gamma_0$ to the interval $[t_1, 1],$
defines a path joining $a_1\in A$ to $b\in {\mathcal S}(B).$ Hence there exists
$s_2\in \,]t_1, 1[\,$ with
$\gamma_0(s_2) = b_2\in B.$ As above, we can also observe that
$b_2\not\in {\mathcal S}(A)$ and  $t_1 < s_2 < \tfrac 1 2.$
Proceeding by induction, we obtain a sequence
$$s_1 < t_1 < s_2 < \dots < s_{j} < t_{j} < s_{j+1} < \frac 1 2$$
with $\gamma(t_i) = a_i \in A$ and $\gamma(s_i) = b_i\in B.$
For $t^*= \sup t_n = \sup s_n \in \,]0,\tfrac 1 2],$ we have that
$\gamma_0(t^*)= \sigma(t^*)= \lim a_n = \lim b_n \in A\cap B,$ a contradiction.
The proof is complete.
\end{proof}

\begin{lemma}\label{lem-ABC.5}
Let $X$ be a connected and locally arcwise connected metric space and let
$A, B, C\subseteq X$ be closed and nonempty sets with $A\cap
B=\emptyset.$ Then $C$ cuts the arcs between $A$ and $B$ if and
only if there exists a continuous function $f: X\to {\mathbb R}$ such that
\begin{equation}\label{eq-AB.2}
f(x)\leq 0,\, \forall\, x\in {\mathcal S}(A),\qquad f(x) \geq 0,\, \forall\, x\in {\mathcal S}(B)
\end{equation}
and
\begin{equation}\label{eq-C.2}
C = \{x\in X: f(x) = 0\}.
\end{equation}
\end{lemma}
\begin{proof}
Clearly, if there exists a continuous function $f: X\to {\mathbb R}$ satisfying \eqref{eq-AB.2}
and \eqref{eq-C.2}, then \eqref{eq-AB.1} and \eqref{eq-C.1} hold too. Hence Lemma \ref{lem-ABC.1}
implies that $C$ cuts the arcs between $A$ and $B.$

Conversely, if $C$ cuts the arcs between $A$ and $B,$ then,
by Lemma \ref{lem-ABC.4},
$C$ cuts the arcs between ${\mathcal S}(A)$ and ${\mathcal S}(B)$ as well.
Therefore we can apply Lemma \ref{lem-ABC.1} with respect to the triple
$({\mathcal S}(A),{\mathcal S}(B),C).$ In particular, the function
$f$ will be defined as in \eqref{eq-f.3},
with $\rho$ like in \eqref{eq-f.1} and $\mu: X \to \{-1,0,1\}$ as follows:
\begin{equation}\label{eq-f.2b}
\mu(x):=
\quad
\left
\{
\begin{array}{ll}
~\,0&\; \mbox{ if } \, x\in C,\\
\!-1 &\; \mbox{ if } \, x\not\in C \;
\mbox{and there exists a path $\gamma_x:[0,1]\to X\setminus C$} \\
~\quad &\mbox{
\,such that $\gamma_x(0)\in {\mathcal S}(A)$ and $\gamma_x(1) = x,$}\\
~\,1&\; \mbox{ elsewhere.}\\
\end{array}
\right.
\end{equation}
\end{proof}

{
\includegraphics[width=3in,height=2.9in]{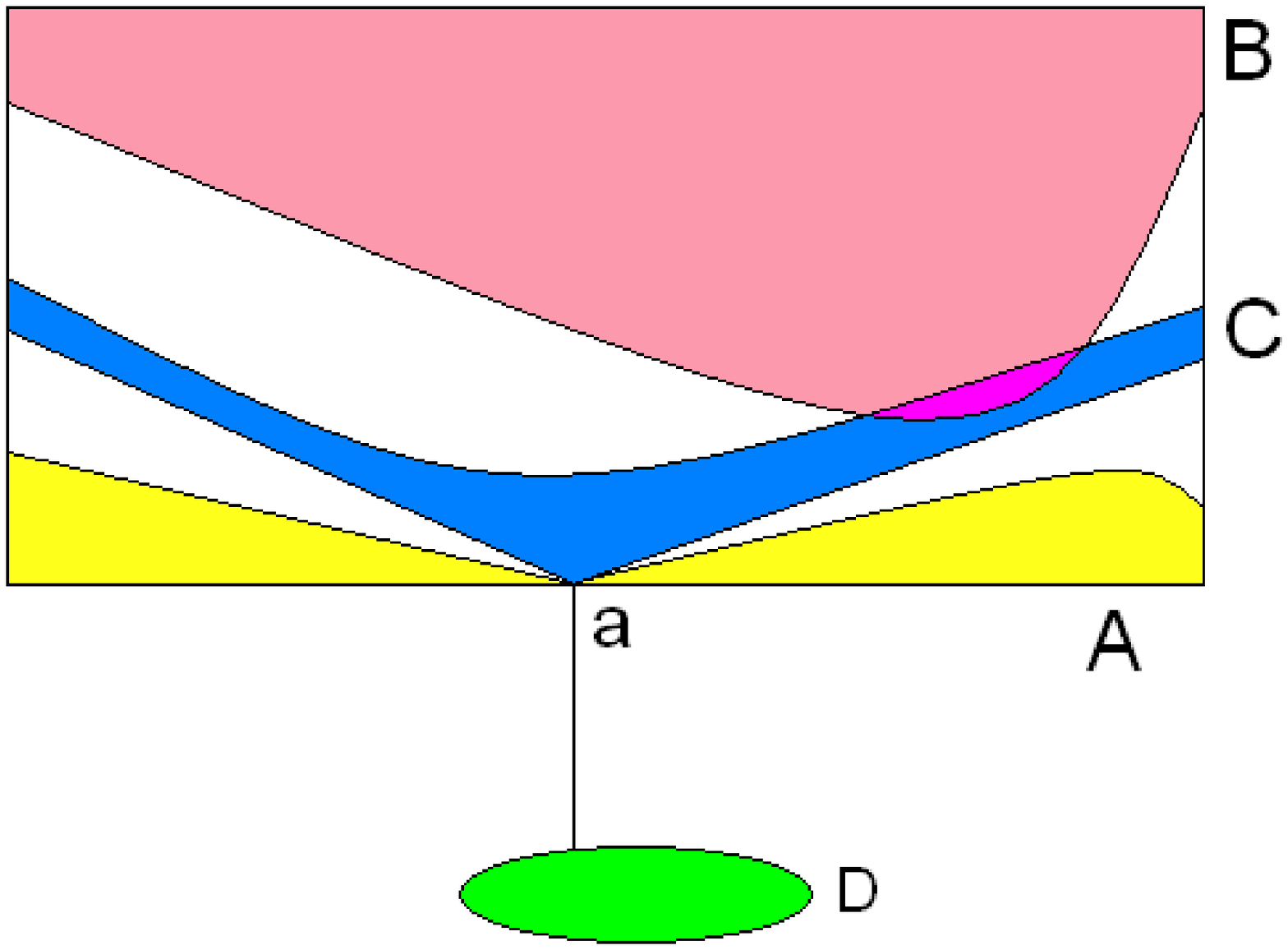}
\includegraphics[width=3in,height=2.9in]{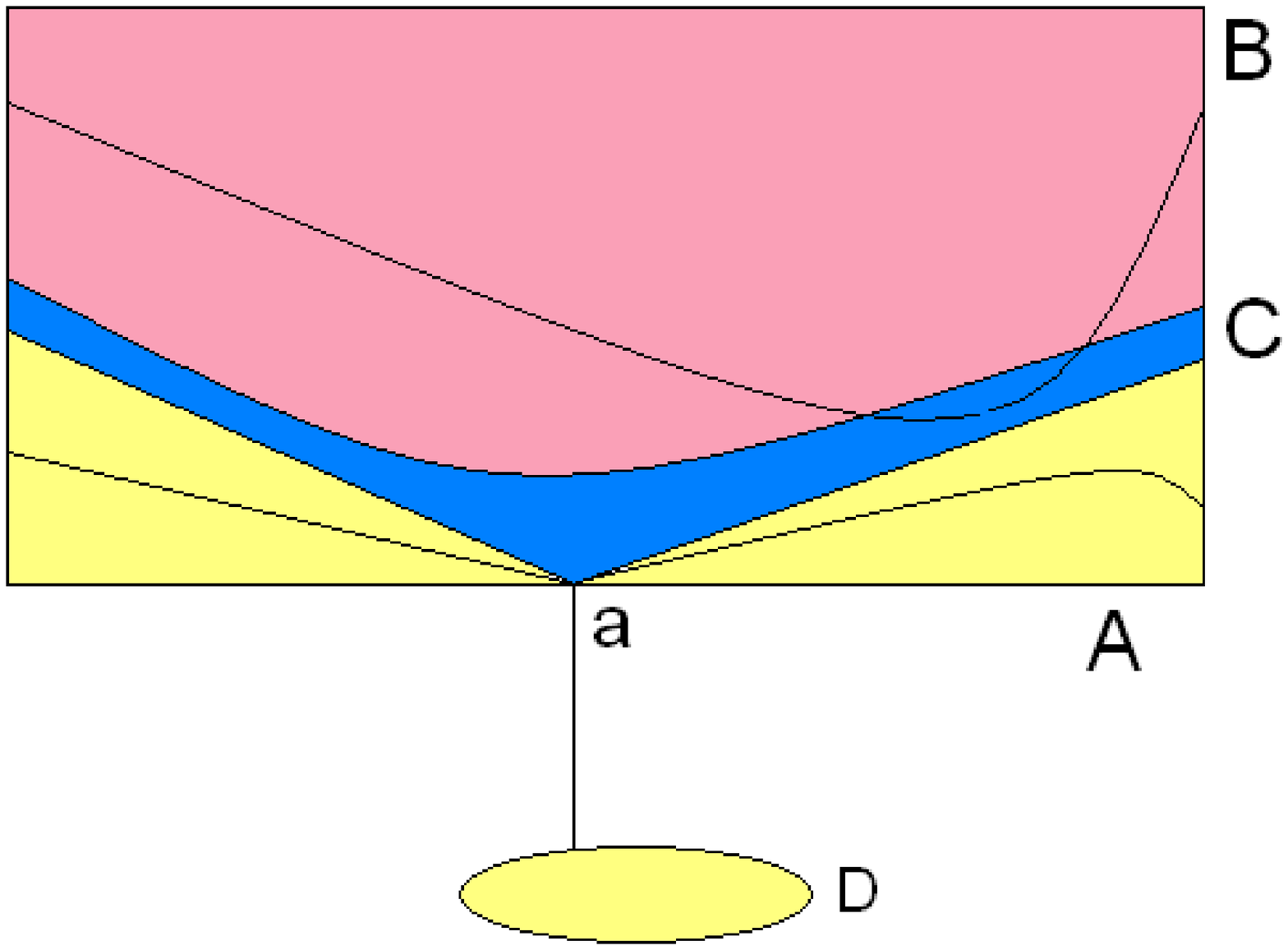}
}
\begin{quote}
{\small{Figure 2.
The figure on the left-hand side gives an illustration
of the situation described in Definition \ref{def-2.1}.
In the present example
the space $X$ is the figure itself as a subset of the plane.
The set $C$ (darker region) cuts the arcs between $A$ (light)
and $B$ (grey). We allow a nonempty intersection between $B$ and $C$
as well as between $A$ and $C$ (the singleton $\{a\}$).
Notice that the only manner to connect with a path
the points of $A$ to the points of the ``appendix'' $D$ is passing through the point $a\in A\cap C.$
\\
The figure on the right-hand side provides an interpretation of Lemma \ref{lem-ABC.5}.
With respect to a function $f$ having its factor $\mu$
defined like in \eqref{eq-f.2b}, we have painted
with a light color the points where $f<0$ and in grey color the points where
$f > 0.$ Note that the region $D$ has been painted in light color, because
$D\subseteq {\mathcal S}(A).$
}
}
\end{quote}

Until now we have considered only the case of paths connecting two disjoint sets $A$ and $B.$
This choice is motivated
by the foregoing examples for subsets of $N-$dimensional spaces.
For sake of completeness we end this section by discussing
the situation in which $A$ and $B$ are joined by a continuum. We confine ourselves to the
following lemma which will find an application in Theorem \ref{th-3c.1}.

\begin{lemma}\label{lem-ABC.6}
Let $X$ be a connected and locally arcwise connected metric space and let
$A, B\subseteq X$ be closed and nonempty sets with $A\cap
B=\emptyset.$ Let $\Gamma\subseteq X$ be a compact connected set
such that
$$\Gamma\cap A\not=\emptyset,\quad \Gamma\cap B\not=\emptyset.$$
Then, for every $\varepsilon > 0$ there exists a path
$\gamma=\gamma_{\varepsilon}\,: [0,1]\to X$ with
$\gamma(0)\in A,$ $\gamma(1)\in B$ and
$${\overline{\gamma}} \subseteq B(\Gamma,\varepsilon).$$
Moreover, if $X$ is locally compact and $C\subseteq X$
is a closed set which cuts the arcs between $A$ and $B,$
then
$$\Gamma\cap C\not=\emptyset.$$
\end{lemma}
\begin{proof}
Let $\varepsilon > 0$ be fixed and consider, for every $p\in \Gamma,$
a radius $\delta_p\in\, ]0,\varepsilon[$ such that any two points
in $B(p,\delta_p)$ can be joined by a path in $B(p,\varepsilon).$
By the compactness of $\Gamma$ we can find a finite number of points
$$p_1\,,p_2\,,\dots, p_k\,\in \Gamma,$$
such that
$$\Gamma\subseteq
{\mathcal B}:=\bigcup_{i=1}^k B(p_i,\delta_i)
\subseteq B(\Gamma,\varepsilon), \qquad \mbox{with } \;\delta_i:= \delta_{p_i}\,.$$
As a consequence of the hypothesis of connectedness of $\Gamma,$
the following property holds:
For every partition of $\{1,\dots,k\}$ into two nonempty disjoint subsets $J_1$
and $J_2\,,$ there exist $i\in J_1$ and $j\in J_2$ such that
$B(p_i,\delta_i)\cap B(p_j,\delta_j)\not=\emptyset.$
This, in turns, implies that we can rearrange the $p_i$'s (possibly changing their
order in the labelling) so that

$$B(p_i,\delta_i)\cap B(p_{i+1},\delta_{i+1})\not=\emptyset,\quad\forall\, i=1,\dots, k-1.$$

Hence we can conclude that for any pair of points $(w,z)\in
{\mathcal B},$ with $w\not= z,$ there is a path $\gamma=
\gamma_{w,z}$ joining $w$ with $z$ and such that
$\overline{\gamma}\subseteq B(\Gamma,\varepsilon).$ In particular,
taking $a\in A\cap \Gamma$ and $b\in B\cap \Gamma,$ we have that
there exists a path $\gamma = \gamma_{\varepsilon}\,: [0,1]\to
B(\Gamma,\varepsilon),$ with $\gamma(0) = a$ and $\gamma(1) = b$
and this proves the first part of the statement.

Assume now that $X$ is locally compact (i.e., for any $p\in X$ and
$\eta >0,$ there exists $0<\mu_p \leq\eta$ such that
${\overline{B(p,\mu_p)}}$ is compact). By the compactness of
$\Gamma$ we can find a finite number of points
$q_1\,,q_2\,,\dots, q_{l}\,\in \Gamma$
and corresponding radii $\mu_i := \mu_{q_i}$ such that
$$\Gamma\subseteq
{\mathcal A}:=\bigcup_{i=1}^l B(q_i,\mu_i)$$
and
${\overline {\mathcal A}} = \bigcup_{i=1}^l {\overline{B(q_i,\mu_i)}}$
is compact. Since ${\mathcal A}$ is an open neighborhood of the compact set $\Gamma,$
there exists $\varepsilon_0 > 0$ such that
$B(\Gamma,\varepsilon_0) \subseteq {\mathcal A}.$ Hence, for each
$0 < \varepsilon \leq \varepsilon_0$ we have that the set
${\overline{B(\Gamma,\varepsilon)}}$ is compact.

Taking now $\varepsilon=\frac 1 n,$ we know that for every $n\in
\mathbb N$ there exists a path $\gamma_{n}\,: [0,1]\to X,$ with
$\gamma_n(0)\in A,$ $\gamma_n(1)\in B$ and $\overline{\gamma_n}
\subseteq B(\Gamma,\tfrac 1 n).$ But, since $C$ cuts the arcs
between $A$ and $B,$ it follows that for every $n\in \mathbb N,$
there is a point $c_n\in C\cap B(\Gamma,\frac 1 n).$
For $n\geq \hat n$ large enough ($\hat n > 1/\varepsilon_0$),
the sequence $(c_n)_{n\geq\hat n}$ is
contained in the compact set ${\overline{B(\Gamma,\varepsilon_0)}}$
and therefore it admits a
converging subsequence
$c_{n_k}\to c^{*}\in {\overline{B(\Gamma,\varepsilon_0)}}.$
Since $d(c_{n_k},\Gamma)<\frac 1 {n_k}$ and
the sets $C,\,\Gamma$ are closed, the limit point $c^{*}\in
\Gamma\cap C.$ The proof is complete.
\end{proof}

\section{Fixed points in generalized rectangles}\label{sec-3a}

We present here some
applications of the topological lemmas obtained in Section \ref{sec-2} to the
intersection of generalized surfaces which
separate the opposite edges of an $N-$dimensional cube. Such
generalized surfaces (see Definition \ref{def-2b.1}) will be described as zero-sets of continuous
scalar functions and therefore a nonempty intersection will be
obtained as a zero of a suitably defined vector field. To this
aim, we recall a classical result about the existence of zeros for
continuous maps in ${\mathbb R}^N,$ the Poincar\'{e}$-$Miranda
theorem.

\begin{theorem}\label{th-2b.1}
Let $I^N:= [0,1]^N$ be the $N-$dimensional unit cube in ${\mathbb R}^N$
for which we denote by
$[x_i = k]:= \{x=(x_1,\dots,x_N)\in I^N: \, x_i = k\}.$
Let
$F = (F_1,\dots, F_N): I^N\to {\mathbb R}^N$
be a continuous mapping such that, for each $i\in\{1,\dots,N\},$
$$F_i(x)\leq 0,\;\forall\, x\in [x_i = 0]\,\mbox{ and }\;
F_i(x)\geq 0,\;\forall\, x\in [x_i = 1]$$
or
$$F_i(x)\geq 0,\;\forall\, x\in [x_i = 0]\,\mbox{ and }\;
F_i(x)\leq 0,\;\forall\, x\in [x_i = 1].$$
Then there exists $\bar{x}\in I^N$ such that $F(\bar{x}) = 0.$
\end{theorem}

We introduce now the spaces we are going to consider.

\begin{definition}\label{def-2b.1}
Let $Z$ be a metric space and
$$h: {\mathbb R}^N\supseteq I^N \to X\subseteq Z$$
be a homeomorphism of $I^N$ onto its image $X.$
We call the pair
$${\widehat{X}}:= ({X},h)$$
a \textit{generalized $N-$dimensional rectangle}
(or, simply, a generalized rectangle) \textit{of $Z$}.
We also set
$${X}_{i}^{\ell}:= h([x_i = 0]),\quad
{X}_{i}^{r}:= h([x_i = 1])$$ and call them the \textit{left} and
the \textit{right} $i-$faces of $X$.
\\
Finally, we define
$$\vartheta X:= h(\partial I^N)$$
and call it the \textit{contour} of $X.$
\end{definition}

\medskip

Our main result is the following theorem which can be considered as a variant of
the Hurewicz$-$Wallman
lemma about dimension \cite{HuWa-41}. The statements of the two results are in fact very similar, but the lemma in \cite{HuWa-41} concerns, instead of our concept of cutting, the stronger idea of separation and requires the sets  $A, B, C$ in Definition \ref{def-2.1} to be pairwise disjoint (see \cite{Ku-68}).
Hence, because of some technical differences
which are crucial in view of our applications, we have chosen to provide all the details.

\begin{theorem}\label{th-2b.2}
Let ${\widehat{X}}:= ({X},h)$ be a generalized rectangle in a metric space $Z.$
Assume that, for each $i\in\{1,\dots,N\},$ there exists a compact set
$${\mathcal S}_i\subseteq {X}$$
such that ${\mathcal S}_i$ cuts the arcs between ${X}_{i}^{\ell}$
and ${X}_{i}^{r}$ in $X.$ Then
$$\bigcap_{i=1}^N {\mathcal S}_i \not=\emptyset.$$
\end{theorem}

\begin{proof}
Through the inverse of the homeomorphism
$h: {\mathbb R}^N\supseteq I^N \to X\subseteq Z$
we can define the compact sets
$$C_i:= h^{-1}({\mathcal S}_i),$$
which cut the arcs between
$[x_i=0]$ and $[x_i=1]$ in $I^N$ (for $i=1,\dots, N$).
Clearly, it will be sufficient to prove that
$$\bigcap_{i=1}^N C_i \not=\emptyset.$$
By Lemma \ref{lem-ABC.1}, for every $i=1,\dots,N,$ there exists
a continuous function $f_i: I^N\to {\mathbb R}$ such that
$f_i \leq 0$ on $[x_i =0]$ and $f_i \geq 0$ on $[x_i=1].$ Moreover,
$$C_i =\{x\in I^N:\, f_i(x) = 0\}.$$
The continuous vector field
$f^{^{^{{\!\!\!\!\rightarrow}}}}:= (f_1,\dots,f_N):I^N\to {\mathbb R}^N$
satisfies the assumptions of the Poincar\'{e}$-$Miranda theorem and therefore
there exists ${\bar{x}}\in I^N$ such that
$$f_i({\bar{x}})=0,\;\;\forall\, i=1,\dots,N.$$
Hence ${\bar{x}}\in \bigcap_{i=1}^N C_i\,.$
\end{proof}

\begin{remark}\label{rem-2b.1}
A very special case in Definition \ref{def-2b.1} occurs
when $Z= {\mathbb R}^N,$ $X= I^N$ and $h= Id_{{\mathbb R}^N}\,.$
In order to avoid a cumbersome notation, we denote the pair
$(I^N,Id_{{\mathbb R}^N})$ simply by $I^N.$
This is, for example, the setting in the Hurewicz$-$Wallman
lemma \cite{HuWa-41} and in the work by Kampen \cite{Ka-00}.

\hfill$\lhd$\\
\end{remark}

\noindent
As a first application of Theorem \ref{th-2b.2} we present a corollary which generalizes
a result due to J. Kampen in \cite{Ka-00}, providing also a more direct proof.

\bigskip

\noindent
\begin{corollary}\label{cor-2b.1} (see \cite[Corollary 4, p.513]{Ka-00})
Let
$$\phi=(\phi_1,\dots,\phi_N): {\mathbb R}^N\supseteq I^N \to {\mathbb R}^N$$
be a continuous map such that for every
$j\in\{1,\dots,N\}$ one of the following conditions holds:
\begin{itemize}
\item[$(C)\;$]
$\phi_{j}([x_j = 0]\cup [x_j = 1])\subseteq [0,1];$\\
\item[$(E)\;$]
for every continuous path $\gamma=(\gamma_1,\dots,\gamma_N):[0,1]\to I^N$
such that $\gamma_j(0) = 0$ and $\gamma_j(1) =1,$ it holds that
$\phi_j({\overline{\gamma}})\supseteq [0,1].$
\end{itemize}
Then $\phi$ has at least a fixed point in $I^N.$
\end{corollary}

If condition $(C)$ holds for some $j\in\{1,\dots,N\},$ we say that $j$ is a
{\textit{contractive direction}},
while we say that $j$ is an {\textit{expansive direction}} when $(E)$ is satisfied.
Once for all, we point out that the term ``contractive'' has to be meant in a broad
manner as it does not imply that the map is a contraction in the classical sense.

\begin{proof}
For every $i\in\{1,\dots,N\}$ we define
$${\mathcal S}_i:= \{x=(x_1,\dots,x_N)\in I^N:\, x_i = \phi_i(x)\}.$$
Let $j\in \{1,\dots,N\}$ be fixed
and let $\gamma: [0,1]\to I^N$ be a continuous map such that $\gamma(0)\in [x_j = 0]$
and $\gamma(1)\in [x_j=1].$

If $j$ is a contractive direction, so that $(C)$ holds,
we have that $\phi_j(\gamma(0)) \geq 0 = \gamma_j(0)$
and $\phi_j(\gamma(1)) \leq 1 = \gamma_j(1).$ Bolzano theorem ensures the existence of $t^*\in [0,1]$
such that $\phi_j(\gamma(t^*)) = \gamma_j(t^*),$ that is $\gamma(t^*)\in {\mathcal S}_j\,.$

On the other hand, if $j$ is an expansive direction and thus
$(E)$ holds, there exist $t_1,t_2\in [0,1]$
such that $\phi_j(\gamma(t_1)) = 0 \leq \gamma_j(t_1)$ as well as
$\phi_j(\gamma(t_2)) = 1 \geq \gamma_j(t_2).$
Bolzano theorem ensures the existence of ${\tilde{t}}\in [0,1]$
(with $t_1 \leq {\tilde{t}} \leq t_2$ or $t_2 \leq {\tilde{t}} \leq t_1$)
such that $\phi_j(\gamma({\tilde{t}})) = \gamma_j({\tilde{t}}),$
that is $\gamma({\tilde{t}})\in {\mathcal S}_j\,.$

The assumptions of Theorem \ref{th-2b.2} are thus satisfied
with respect to $X = I^N$ and $h=Id_{{\mathbb R}^N}\,$
and so
$\bigcap_{i=1}^N {\mathcal S}_i \not=\emptyset.$
By definition, any point
${\bar{x}}\in \bigcap_{i=1}^N {\mathcal S}_i$ is such that $\phi({\bar{x}}) = {\bar{x}}.$
\end{proof}

\bigskip

\noindent
Corollary \ref{cor-2b.1} extends \cite[Corollary 4, p.513]{Ka-00} where, instead of
$(C),$ the stronger condition
\begin{itemize}
\item[$(C')\;$]
$\phi_{j}(I^N)\subseteq [0,1]$
\end{itemize}
was assumed.

\bigskip

\noindent In order to show the relationship with \cite{Zg-01}, we
introduce the following definition, inspired by \cite[Definition 4]{Zg-01}.

\begin{definition}\label{def-2b.2}
Assume we have two $N-$dimensional rectangles
${\mathcal R}_1:= \prod_{i=1}^N [a_i,b_i]$ and
${\mathcal R}_2:= \prod_{i=1}^N [c_i,d_i]$
and let $\psi: {\mathbb R}^N \supseteq{\mathcal R}_1\to {\mathbb R}^N$ be a
continuous map.
Let $1\leq i_1 < i_2 <\dots < i_k \leq N$ be a finite sequence of indexes.
We say that {\textit{${\mathcal R}_1$ $\psi-$covers ${\mathcal R}_2$ in
$(i_1,i_2,\dots,i_k)-$direction}} if the following conditions hold:
\begin{itemize}
\item{}\:
for every $j\in \{i_1,i_2,\dots,i_k\},$
$$
[c_j,d_j] \subseteq \Bigl[\,\max_{x\in {\mathcal R}_1\,,\, x_j = a_j}\psi_j(x),
\min_{x\in {\mathcal R}_1\,,\, x_j = b_j}\psi_j(x)\,\Bigr]
$$
or
$$
[c_j,d_j] \subseteq \Bigl[\,\max_{x\in {\mathcal R}_1\,,\, x_j = b_j}\psi_j(x),
\min_{x\in {\mathcal R}_1\,,\, x_j = a_j}\psi_j(x)\,\Bigr]\,;
$$
\item{}\;
for every $j\not\in \{i_1,i_2,\dots,i_k\},$
$\psi_j({\mathcal R}_1)\subseteq [c_j,d_j].$
\end{itemize}
\end{definition}

\bigskip

\noindent
\begin{corollary}\label{cor-2b.2} (see \cite{Zg-01})
Let
$\psi: {\mathbb R}^N\supseteq {\mathcal R} \to {\mathbb R}^N$ be a continuous
map, for
${\mathcal R}:= \prod_{i=1}^N [a_i,b_i]$ and
suppose there exists a
finite sequence of indexes
$1\leq i_1 < i_2 <\dots < i_k \leq N,$ such that
${\mathcal R}$ $\psi-$covers ${\mathcal R}$ in
$(i_1,i_2,\dots,i_k)-$direction. Then
$\psi$ has at least a fixed point in ${\mathcal R}.$
\end{corollary}
\begin{proof}
The homeomorphism $h=(h_1,\dots,h_N): {\mathbb R}^N\to {\mathbb R}^N,$
with\begin{equation}\label{eq-2b.1}
h_i(x_1,\dots,x_N):= a_i + (b_i - a_i) x_i
\end{equation}
maps the unitary cube $I^N$ onto ${\mathcal R}.$
It is straightforward to check that the map $\phi:= h^{-1}\circ\psi\circ h$
satisfies assumption $(E)$ along the components in $\{i_1,i_2,\dots,i_k\},$
while condition $(C')$ holds along the remaining components.
Hence Corollary \ref{cor-2b.1} applies ensuring the existence of
at least a fixed point
${\bar{x}}\in I^N$ for the map $\phi.$ This implies that
${\bar{y}}:=h({\bar{x}})\in {\mathcal R}$ is a fixed point for $\psi.$
\end{proof}

Corollary \ref{cor-2b.2} is a trivial case of a widely more
general result (by P. Zgliczy\'nski in \cite{Zg-01}) that we recall
below as Theorem \ref{th-2b.z}.
Actually, the author considered a more restrictive covering condition
(i.e. {\textit{covering with margin $\delta$}}) for maps defined on the
whole ${\mathbb R}^N$ in order to apply his main
result also to the case of small perturbations of a given map.

\bigskip

\noindent
\begin{theorem}\label{th-2b.z} (see \cite[Theorem 1, p.1042]{Zg-01})
Suppose we have a family of $N-$dimensional rectangles ${\mathcal
R}_l:= \prod_{i=1}^N [a_i^{(l)},b_i^{(l)}]$ and a family of
continuous maps $\psi_l: {\mathcal R}_l\to {\mathbb R}^N,$ for
$l=0,\dots,m -1.$ Assume there exists a finite sequence of indexes
$1\leq i_1 < i_2 <\dots < i_k \leq N,$ such that for $l=0,\dots,m
-1,$ ${\mathcal R}_l$ $\psi_l-$covers ${\mathcal R}_{l+1}$
$(\mbox{mod }m)$ in $(i_1,i_2,\dots,i_k)-$direction. Then there
exists $w\in {\mathcal R}_0$ such that
\begin{itemize}
\bigskip
\item[] $\psi_l\circ\psi_{l-1}\circ\dots\circ \psi_0(w) \in
{\mathcal R}_{l+1}\,,\quad$
for $ l= 0,1,\dots, m-2\,;$\\
\item[] $\psi_{m -1}\circ\psi_{m -2}\circ\dots\circ \psi_0(w) =
w.$
\end{itemize}
\end{theorem}

\noindent
Our goal is to obtain a result which is closely
related to Theorem \ref{th-2b.z}, but exploiting an expansive
condition along the paths like in Corollary
\ref{cor-2b.1} and in \cite{Ka-00}.
To this purpose we first introduce the following definition:

\begin{definition}\label{def-2b.3}
Assume we have two $N-$dimensional rectangles
${\mathcal R}_1:= \prod_{i=1}^N [a_i,b_i]$ and
${\mathcal R}_2:= \prod_{i=1}^N [c_i,d_i]$
and let $\psi: {\mathbb R}^N \supseteq{\mathcal R}_1\to {\mathbb R}^N$ be a
continuous map.
Let $1\leq i_1 < i_2 <\dots < i_k \leq N$ be a finite sequence of indexes.
We say that {\textit{${\mathcal R}_1$ $\psi-$covers ${\mathcal R}_2$ in
$(i_1,i_2,\dots,i_k)-$direction along the paths}} if the following conditions hold:
\begin{itemize}
\item{}\:
for every $j\in J:=\{i_1,i_2,\dots,i_k\}$ and
every continuous path $\gamma=(\gamma_1,\dots,\gamma_N):[0,1]\to {\mathcal R}_1$
satisfying $\gamma_j(0) = a_j$ and $\gamma_j(1) =b_j\,,$ it holds that
$\psi_j({\overline{\gamma}})\supseteq [c_j,d_j]\,;$\\
\item{}\;
for every $j\not\in J,$
$\psi_j({\mathcal R}_1)\subseteq [c_j,d_j].$
\end{itemize}
\end{definition}

\begin{remark}\label{rem-2b.kz}
We observe that Definition \ref{def-2b.2} and Definition \ref{def-2b.3}
coincide in the one$-$dimensional case, while they differ for $N\geq 2$
if $J\not=\emptyset$ (that is, if at least one expansive direction is present).
To be more precise, it is straightforward to verify that any map $\psi$
satisfying Definition \ref{def-2b.2}, with respect to a pair of
$N-$dimensional rectangles, fulfills also Definition \ref{def-2b.3}.
On the other hand, we give an example (see Example \ref{ex-2b.1} below)
of a two$-$dimensional map
which satisfies the latter definition, but not the former.
This shows that, in principle, Corollary \ref{cor-2b.1} is more general than
Corollary \ref{cor-2b.2}.\hfill$\lhd$
\end{remark}

\begin{example}\label{ex-2b.1}
Let $\phi = (f,g): {\mathbb R}^2 \to {\mathbb R}^2$ be the continuous map defined by
$$
\begin{array}{ll}
{\displaystyle{
f(x,y):= \frac{1}{2} + c \cos( 2\pi (k x + \ell (y - \tfrac{1}{2}) ) ),}}\\
\\
{\displaystyle{g(x,y):= \frac{1}{2} + d \sin(2\pi (y + m x) ),}}
\end{array}
$$
where $c,d \in {\mathbb R}$ and $k,l,m\in {\mathbb N}$ are chosen in order to satisfy
$$0 < d \leq \frac{1}{2} < c,\;\; \ell > \frac{1}{d}\,,\;\; k \geq  \ell +1,\;\; m\geq 1.$$
By the above positions, we see immediately that
$$g(x,y)\in \Bigl [ \frac 1 2 - d, \frac 1 2 + d\Bigr] \subseteq [0,1],\;\;\forall\, (x,y)\in I^2,$$
so that $\phi$ is contractive in the second component.
We prove now that $\phi$ is expansive (along the paths) with respect to the first component.
More precisely, we check that, for $J=\{1\},$ $I^2$ $\phi-$covers $I^2$ in the
$x-$direction along the paths.
\\
Let $\gamma = (\gamma_1,\gamma_2): [0,1]\to I^2$ be a continuous map such that
$\gamma_1(0) = 0$ and $\gamma_1(1)= 1.$ We claim that for
$F(t):= f(\gamma_1(t),\gamma_2(t)),$
$$F([0,1]) \supseteq [0,1]$$
follows. In fact, the set $\{2\pi k\gamma_1(t): \, t\in [0,1]\}$ coincides with the interval
$[0,2\pi k],$ while the set $\{2\pi \ell\gamma_2(t) - \pi \ell: \, t\in [0,1]\}$
is contained in the interval
$[-\pi\ell,\pi \ell],$ so that the set
$\{2\pi (k \gamma_1(t) + \ell (\gamma_2(t) - \tfrac{1}{2}))\,:\,
t\in [0,1]\}$ contains the interval $[\pi \ell, 2\pi k -\pi\ell]$
whose length is at least $2\pi.$
Therefore,
$F([0,1])\supseteq \bigl[ \frac 1 2 - c, \frac 1 2 + c\bigr] \supseteq [0,1]$
and thus we have proved that $\phi$ agrees with Definition \ref{def-2b.3}.
\\
On the other hand, it is not true that $I^2$ $\phi-$covers $I^2$ in the
$x-$direction.
In fact, for $x=0$ or $x=1,$
$$
[0,1]\nsubseteq
\Bigl[\,\max_{(0,\,y)\in I^2} f(0,y),
\min_{(1,\,y)\in I^2} f(1,y)\,\Bigr],\quad
[0,1]\nsubseteq
\Bigl[\,\max_{(1,\,y)\in I^2} f(1,y),
\min_{(0,\,y)\in I^2} f(0,y)\,\Bigr].
$$
In fact, it is not difficult to prove even more, that is, for every ${\hat{x}} \in [0,1]$
the set
$\{ f({\hat{x}},y): \, y\in [ \tfrac 1 2 - d, \tfrac 1 2 + d ]\,\}$
covers $[0,1]$ and therefore, it cannot lie at the left or at the right of the
interval $[0,1].$ More precisely, for every ${\hat{x}} \in [0,1]$ it holds that:
$$
\begin{array}{ll}
\min \{f({\hat{x}},y): \, y\in [ \tfrac 1 2 - d, \tfrac 1 2 + d ]\,\}= \frac 1 2 - c<0,\\
\\
\max \{f({\hat{x}},y): \, y\in [ \tfrac 1 2 - d, \tfrac 1 2 + d ]\,\} = \frac 1 2 + c>1.
\end{array}$$
This shows that there is no sub$-$rectangle of the form
${\mathcal R}:=[a,b]\times [ \tfrac 1 2 - d, \tfrac 1 2 + d ]$ such that
${\mathcal R}$ $\phi-$covers $I^2$ in the $x-$direction. \hfill$\lhd$
\end{example}

\noindent
Next, we provide an improvement of Corollary \ref{cor-2b.1} which will be our main tool
in the proof of Theorem \ref{th-2b.4} below.

\begin{corollary}\label{cor-2b.3}
Let ${\mathcal R}:= \prod_{i=1}^N [a_i,b_i]$ be an $N-$dimensional rectangle
with opposite $i-$faces
$${\mathcal R}^{\ell}_i := \{x\in {\mathcal R}: \, x_i = a_i\},\quad
{\mathcal R}^{r}_i := \{x\in {\mathcal R}: \, x_i = b_i\}$$
and let
$\phi=(\phi_1,\dots,\phi_N): {\mathbb R}^N\supseteq {\mathcal R} \to {\mathbb R}^N$
be a continuous map.
Suppose there exists $J\subseteq \{1,\dots,N\}$ such that
the following conditions hold:\\
\begin{itemize}
\item[$(C)\;$] $\phi_{j}({\mathcal R}^{\ell}_j\cup {\mathcal
R}^{r}_j)
\subseteq [a_j,b_j],\;\;\forall\, j\not\in J;$\\
\item[$(E_{W})\;$]
for each $j\in J$ there is a (nonempty) compact set $W_j\subseteq {\mathcal R}$
such that for
every continuous path $\gamma=(\gamma_1,\dots,\gamma_N):[0,1]\to {\mathcal R}$
satisfying $\gamma_j(0) = a_j$ and $\gamma_j(1) =b_j\,,$ there exists a sub$-$path
$\sigma$ with ${\overline{\sigma}}\subseteq W_j$ and
$\phi_j({\overline{\sigma}})\supseteq [a_j,b_j].$\\
\end{itemize}
Then $\phi$ has at least a fixed point in $W:= \bigcap_{j\in J} W_j\,.$
{\footnote{$\,$ The role of $W$ in the theorem is meaningful only if $(E_{W})$ holds
for at least one component $j.$ Otherwise, if $(C)$ holds for every $j=1,\dots,N,$
we read the theorem for $W={\mathcal R}.$}}
\end{corollary}

\begin{proof}
We follow the argument already described along the proof of Corollary \ref{cor-2b.1}.
We also assume that there exists at least one $j\in\{1,\dots,N\}$
for which $(E_{W})$ is satisfied (otherwise the result follows directly from
Rothe fixed point theorem).
\\
For every $j\not\in J,$ we define
$${\mathcal S}_j:= \{x=(x_1,\dots,x_N)\in {\mathcal R}:\, x_j = \phi_j(x)\},$$
while, for $j\in J,$ we set
$${\mathcal S}_j:= \{x=(x_1,\dots,x_N)\in W_j\,:\, x_j = \phi_j(x)\}.$$
Let $j\in \{1,\dots,N\}$ be fixed and let $\gamma: [0,1]\to
{\mathcal R}$ be a continuous map such that $\gamma(0)\in
{\mathcal R}^{\ell}_j$ and $\gamma(1)\in {\mathcal R}^{r}_j.$

If $j$ is a contractive direction, so that $(C)$ holds,
by repeating exactly the same argument as in the
proof of Corollary \ref{cor-2b.1}, we find that there exists
$t^*\in [0,1]$
such that $\phi_j(\gamma(t^*)) = \gamma_j(t^*),$
that is $\gamma(t^*)\in {\mathcal S}_j\,.$

On the other hand, if $j$ is an expansive direction so that
$(E_{W})$ holds, there exist $0\leq t_1 < t_2\leq 1$
such that
$\gamma(t)\in W_j\,,$ for every $t\in [t_1,t_2]$ and,
moreover,
$\phi_j(\gamma(t_1)) = a_j \leq \gamma_j(t_1)$ as well as
$\phi_j(\gamma(t_2)) = b_j \geq \gamma_j(t_2)$
(or $\phi_j(\gamma(t_1)) = b_j \geq \gamma_j(t_1)$ as well as
$\phi_j(\gamma(t_2)) = a_j \leq \gamma_j(t_2)$).
Bolzano theorem ensures the existence of ${\tilde{t}}\in [t_1,t_2]$
such that $\phi_j(\gamma({\tilde{t}})) = \gamma_j({\tilde{t}})$
and also $\gamma({\tilde{t}})\in W_j\,.$ Hence,
$\gamma({\tilde{t}})\in {\mathcal S}_j\,.$

The assumptions of Theorem \ref{th-2b.2} are thus satisfied
with respect to $X = {\mathcal R}$ and $h$
defined as in \eqref{eq-2b.1}.
Therefore,
$\bigcap_{i=1}^N {\mathcal S}_i \not=\emptyset.$
By definition, any point
${\bar{x}}\in \bigcap_{i=1}^N {\mathcal S}_i$ is such that $\phi({\bar{x}}) = {\bar{x}}.$
Moreover ${\bar{x}}\in W.$
\end{proof}

\begin{remark}\label{rem-2b.2}
Clearly, Corollary \ref{cor-2b.3} is still true if we replace hypothesis $(C)$ with
\begin{itemize}
\item[$(C')\;$]
$\phi_{j}({\mathcal R})
\subseteq [a_j,b_j],\;\;\forall\, j\not\in J.$
\end{itemize}
This assumption (which is slightly more restrictive than $(C)$\,)
is crucial when we consider compositions of maps.\hfill$\lhd$
\end{remark}

We are now in position to prove our extension of Theorem \ref{th-2b.z}.
Along the proof, we use the following notation:
\\
Let $\alpha,\beta\in {\mathbb R}$ with $\alpha < \beta.$ We denote by $\eta_{\,[\alpha,\beta]}$
the (continuous) projection of ${\mathbb R}$ onto the interval $[\alpha,\beta],$ defined by
\begin{equation}\label{eq-2b.e}
\eta_{\,[\alpha,\beta]}(s):= \max\{\alpha,\min\{s,\beta\}\,\}.
\end{equation}

\noindent
\begin{theorem}\label{th-2b.4}
Suppose we have a family of $N-$dimensional rectangles ${\mathcal
R}_l:= \prod_{i=1}^N [a_i^{(l)},b_i^{(l)}]$ and a family of
continuous maps $\psi_l: {\mathcal R}_l\to {\mathbb R}^N,$ for
$l=0,\dots,m -1.$ Assume there exists a finite sequence of indexes
$1\leq i_1 < i_2 <\dots < i_k \leq N,$ such that for $l=0,\dots,m
-1,$ ${\mathcal R}_l$ $\psi_l-$covers ${\mathcal R}_{l+1}$
$(\mbox{mod }m)$ in $(i_1,i_2,\dots,i_k)-$direction along the
paths. Then there exists $w\in {\mathcal R}_0$ such that
\begin{eqnarray}
&~&\psi_l\circ\psi_{l-1}\circ\dots\circ \psi_0(w) \in {\mathcal
R}_{l+1}\,,\;\;
\mbox{for } l= 0,1,\dots, m -2\,;\label{eq-2b.2}\\
&~&\psi_{m -1}\circ\psi_{m -2}\circ\dots\circ \psi_0(w) =
w.\label{eq-2b.3}
\end{eqnarray}
\end{theorem}
\begin{proof}
We set
$$J:= \{i_1,\dots,i_k\}.$$
If $J=\emptyset,$ it follows that $\psi_l({\mathcal
R}_l)\subseteq {\mathcal R}_{l+1},$ for $l=0,\dots,m -1$
$(\mbox{mod }m),$ and the result is an immediate consequence of the
Brouwer fixed point theorem. Thus, for the remainder of the proof, we
assume $J\not=\emptyset.$

First of all, we extend the map $\psi_l =
(\psi^{(l)}_1,\dots,\psi^{(l)}_N)$ to a continuous mapping
$${\widetilde{\psi}}_l = ({\widetilde{\psi}\,}^{(l)}_1,\dots,{\widetilde{\psi}\,}^{(l)}_N)
: {\mathbb R}^N \to {\mathbb R}^N,$$
defined by
$${\widetilde{\psi}\,}^{(l)}_{i}(x):= \eta_{\,[a_i^{(l+1)},\, b_i^{(l+1)}]}(\psi^{(l)}_i(
{P_{{\mathcal R}_{l}}}(x)
)),
\quad \forall\, i= 1,\dots,N,$$
where we have called ${P_{{\mathcal R}_{l}}}$ the projection of ${\mathbb R}^N$ onto the
rectangle ${\mathcal R}_{l}\,,$ given by
$${P_{{\mathcal R}_{l}}}(x):= \bigl(
\eta_{\,[a_1^{(l)},\, b_1^{(l)}]}(x_1),\dots, \eta_{\,[a_N^{(l)},\, b_N^{(l)}]}(x_N)
\bigr).$$
Then, for every $j\in J,$ we define
$W_j$ as the set of the points
$x\in {\mathcal R}_0$ satisfying the following conditions:
$$
\begin{array}{ll}
\psi^{(0)}_j(x)\in [a_j^{(1)},\, b_j^{(1)}],\\
\\
\psi^{(1)}_j(\psi_0(x))\in [a_j^{(2)},\, b_j^{(2)}],\\
$~$\qquad\;\vdots\\
\psi^{(m - 2)}_j(\psi_{m - 3}\circ\dots\circ\psi_0)(x)
\in [a_j^{(m -1)},\, b_j^{(m -1)}],\\
\\
\psi^{(m - 1)}_j(\psi_{m - 2}\circ\psi_{m -
3}\circ\dots\circ\psi_0)(x)
\in [a_j^{(0)},\, b_j^{(0)}].\\
\end{array}
$$
We are going to apply Corollary \ref{cor-2b.3} (in the version of Remark
\ref{rem-2b.2}) to the composite map
$$\phi=(\phi_1,\dots,\phi_N):= {\widetilde{\psi}}_{m -1}\circ{\widetilde{\psi}}_{m -2}\circ\dots\circ
{\widetilde{\psi}}_0\,.$$
Notice that $\phi$ is well defined as a continuous map on ${\mathcal R}_0\,.$
Indeed, by the property of the projections $\eta_{\,[a_i^{(l)},\, b_i^{(l)}]}$
we have that
$${\widetilde{\psi}\,}_{l}({\mathcal R}_l)\subseteq {\mathcal R}_{l+1}\,,\;\forall\,
l= 0,\dots, m - 1.$$
As a preliminary observation we note that any
fixed point $w$ for $\phi,$ with $w\in W:= \bigcap_{j\in J}
W_j\,,$ satisfies the conditions \eqref{eq-2b.2} and
\eqref{eq-2b.3}. This follows from the fact that
$$\eta_{\,[a_i^{(l)},\, b_i^{(l)}]}(s) = s,\;\; \mbox{for } s \in [a_i^{(l)},\, b_i^{(l)}]$$
and that, when $j\not\in J,$
$${\widetilde{\psi}\,}^{(l)}_j(x) = \psi^{(l)}_j(x),\quad\forall\, x\in {\mathcal R}_l\,.$$

Let $j\in \{1,\dots,N\}$ be a fixed index. We distinguish two cases:
\\
$(a)\;$ $j\not\in J.\quad$ In this situation, for every
$l=0,\dots, m - 1,$ we have
$${\widetilde{\psi}\,}^{(l)}_j(x) \in [a_j^{(l+1)},\, b_j^{(l+1)}].$$
Therefore, $\phi$ satisfies $(C').$
\\
$(b)\;$ $j\in J.\quad$ Let $\gamma = (\gamma_1,\dots,\gamma_N): [0,1] \to{\mathcal R}_0$ be a continuous path
satisfying $\gamma_j(0) = a^{(0)}_j$ and $\gamma_j(1) = b^{(0)}_j\,.$ Our aim is to prove
that there exists a sub-path $\sigma$ of $\gamma$ with ${\overline{\sigma}}\subseteq W_j$
and such that $\phi_j({\overline{\sigma}})\supseteq [a^{(0)}_j,b^{(0)}_j].$
By the assumptions, we know that $\psi^{(0)}_j({\overline{\gamma}}) \supseteq [a^{(1)}_j,b^{(1)}_j].$
If we consider the continuous real valued function
$$g:[0,1]\to{\mathbb R},\quad
g(t):=\psi^{(0)}_j(\gamma(t)),$$
we can find $0\leq t_1 < t_2\leq 1$ such that $g(t)\in [a^{(1)}_j,b^{(1)}_j],$
for every $t\in [t_1,t_2]$ and, either $g(t_1) = a^{(1)}_j$ and
$g(t_2) = b^{(1)}_j\,,$ or
$g(t_1) = b^{(1)}_j$ and
$g(t_2) = a^{(1)}_j.$ The restriction $\gamma_0:= \gamma|_{[t_1,t_2]}$ is a sub-path
of $\gamma$ such that $\psi^{(0)}_j({\overline{{\gamma}_0}})= [a^{(1)}_j,b^{(1)}_j].$
Moreover, we also have that
$${\widetilde{\psi}\,}^{(0)}_j(\gamma_0(t)) = \psi^{(0)}_j(\gamma_0(t)),\;\;\forall\,
t\in [t_1,t_2].$$
If we like, we can take a continuous path $\sigma_0:[0,1]\to {\mathcal R}_0\,,$
in the same equivalence class of $\gamma_0\,,$ such that
$${\widetilde{\psi}\,}^{(0)}_j(\sigma_0(t)) = \psi^{(0)}_j(\sigma_0(t))\in
[a^{(1)}_j, b^{(1)}_j],\;\;\forall\,
t\in [0,1],$$
$$\psi^{(0)}_j(\sigma_0(0)) = a^{(1)}_j\,,\quad
\psi^{(0)}_j(\sigma_0(1)) = b^{(1)}_j\,.$$
Now, with respect to the path
$$\gamma_1:= {\widetilde{\psi}\,}_{0}\circ \sigma_0: [0,1]\to {\mathcal R}_1\,,$$
we are exactly in the same situation like we were with respect to the
path $\gamma: [0,1]\to {\mathcal R}_0\,.$

At this moment, we can proceed by induction, just repeating a finite number of
times the previous argument, until we find a sub-path $\sigma$ of $\gamma,$
with $\sigma: [0,1]\to {\mathcal R}_0$ satisfying the following conditions:
\begin{eqnarray*}
&~&{\widetilde{\psi}\,}^{(l)}_j({\widetilde{\psi}}_{l - 1}\circ\dots\circ
{\widetilde{\psi}}_0)(\sigma(t)) = \psi^{(l)}_j(\psi_{l -
1}\circ\dots\circ\psi_0)(\sigma(t)) \in [a_j^{(l+1)},\,
b_j^{(l+1)}],\\
&~&\qquad\qquad\qquad\;\;\forall\, t\in [0,1],\;\forall\, l= 0,\dots,m
-1\; (\mbox{mod }m),
\end{eqnarray*}
\vspace{-0.2cm}
$$\psi^{(m - 1)}_j(\psi_{m - 2}\circ\dots\circ\psi_0)(\sigma(0))
= a_j^{(0)},\;\; \psi^{(m - 1)}_j(\psi_{m -
2}\circ\dots\circ\psi_0)(\sigma(1)) = b_j^{(0)}\,.$$ Therefore
${\overline{\sigma}}\subseteq W_j$ and $\phi_j({\overline{\sigma}})=
[a^{(0)}_j,b^{(0)}_j],$ that is, $(E_{W})$ of Corollary
\ref{cor-2b.3} is satisfied.
\\
In this manner we have proved that for every $j\in \{1,\dots,N\}$
either $(C')$ or $(E_{W})$ is fulfilled with respect to $\phi.$
Hence, Corollary \ref{cor-2b.3} ensures the existence of at least a fixed point for $\psi$ in $W:= \bigcap_{j\in J}
W_j\,$ that, as already observed, satisfies conditions \eqref{eq-2b.2} and
\eqref{eq-2b.3}. The proof is complete.
\end{proof}

\section{Continua of fixed points for maps depending on parameters}\label{sec-3b}
In this section we still consider the intersection of generalized surfaces
which separate the opposite edges of an $N-$dimensional cube,
but in the case in which the number of the
cutting surfaces is smaller than the dimension of the space.
Here
our main tool is a result by
Fitzpatrick, Massab\'{o} and  Pejsachowicz (see \cite[Theorem 1.1]{FiMaPe-86})
on the covering dimension of the
zero set of an operator depending on parameters.
For the reader's convenience, we recall the concept of covering dimension
as can be found in \cite{En-78}.

\begin{definition}\label{def-3b.1}\cite[p.54, p.208]{En-78}
Let $Z$ be a metric space. We say that $\mbox{\rm dim \!} Z \leq n$
if every finite open cover of the space $Z$ has a finite open [closed]
refinement of order $\leq n.$
The object $\mbox{\rm dim \!} Z\in {\mathbb N}\cup\{\infty\}$
is called the \textit{covering dimension} or the
{\textit{\v Cech$-$Lebesgue dimension}} of the metric space $Z.$
According to \cite{FiMaPe-86}, if $z_0\in Z,$ we also say that
{\textit{$\mbox{\rm dim \!} Z \geq j$ at $z_0$}} if each neighborhood
of $z_0$ has dimension at least $j.$
\end{definition}

\noindent
The {\textit{order}} of a family ${\mathfrak A}$ of subsets of $Z$
(used in the above definition)
is the largest integer $n$ such that the family ${\mathfrak A}$ contains
$n+1$ sets with a non-empty intersection; if no such integer exists
we say that the family ${\mathfrak A}$ has order infinity.
By a classical result from Topology (see \cite[The coincidence theorem]{En-78})
the covering dimension coincides with the \textit{inductive dimension}
\cite[p.3]{En-78} for separable metric spaces.

\bigskip

We keep the notation of the previous section.
In particular, we recall that if
${\mathcal R}:= \prod_{i=1}^N [a_i,b_i]$ is an $N-$dimensional rectangle, we
denote its opposite $i-$faces by
$${\mathcal R}^{\ell}_i := \{x\in {\mathcal R}: \, x_i = a_i\},\quad
{\mathcal R}^{r}_i := \{x\in {\mathcal R}: \, x_i = b_i\}.$$

\begin{theorem}\label{th-3b.1}
Let ${\mathcal R}:= \prod_{i=1}^N [a_i,b_i]$ be an $N-$dimensional
rectangle and let $P = (p_1,\dots,p_N)$ be any interior point of
${\mathcal R}.$ Let $n\in \{1,\dots, N-1\}$ be fixed. Suppose that
$F = (F_1,\dots, F_n): {\mathcal R}\to {\mathbb R}^n$ is a
continuous mapping such that, for each $i\in\{1,\dots,n\},$
$$F_i(x) < 0,\;\forall\, x\in {\mathcal R}^{\ell}_i\quad\mbox{and }\;
F_i(x) > 0,\;\forall\, x\in {\mathcal R}^{r}_i$$
or
$$F_i(x) > 0,\;\forall\, x\in {\mathcal R}^{\ell}_i\quad\mbox{and }\;
F_i(x) < 0,\;\forall\, x\in {\mathcal R}^{r}_i\,.$$ Define also
the affine map
\begin{equation}\label{eq-3b.1}
\pi: {\mathbb R}^N\to {\mathbb R}^{N-n},\quad
\pi_j(x_1,\dots,x_{N}):= x_j - p_j\,,\; j = N-n + 1,\dots, N.
\end{equation}
Then there exists a connected subset ${\mathcal Z}$ of
$$F^{-1}(0) = \{x\in {\mathcal R}:\, F_i(x) =0,\,\forall\, i=1,\dots,n\}$$
whose dimension at each point is at least
$N-n.$
Moreover,
$$\mbox{\rm dim}({\mathcal Z}\cap\partial {\mathcal R})\geq N - n - 1$$
and
$$\pi: {\mathcal Z}\cap \partial {\mathcal R} \to {\mathbb R}^{N-n}\setminus\{0\}$$
is essential.
\end{theorem}
\begin{proof}
We define the continuous mapping
$$H = (F,\pi):{\mathcal R}\to {\mathbb R}^N.$$
By the assumptions on $F$ and $\pi$ we have
$$\mbox{\rm deg}(H, \,{\mathcal R}^{^{^{\!\!\!\!\!{o}}}}\,,0) = (-1)^d\not=0,$$
where $d$ is the number of components $i\in\{1,\dots,n\}$ such
that $F_i(x) > 0$ for $x \in {\mathcal R}^{\ell}_i$ (and also
$F_i(x) < 0$ for $x\in {\mathcal R}^{r}_i$ ). Hence $\pi$ turns
out to be a complementing map for $F$ (according to
\cite{FiMaPe-86}). A direct application of \cite[Theorem
1.1]{FiMaPe-86}  gives the thesis (note that the dimension $m$ in
\cite[Theorem 1.1]{FiMaPe-86} corresponds to our $N-n$).
\end{proof}

\begin{remark}\label{rem-3b.1}
If $a_i < 0 < b_i$ ($\forall\, i=1,\dots,N$) we can take $P=0,$ so
that the complementing map is just the projection $\pi: {\mathbb
R}^N\to {\mathbb R}^{N-n}\,.$ \hfill$\lhd$
\end{remark}

\noindent
For the reader's convenience, we recall that (according to the definitions in
\cite{FiMaPe-86}), given an open bounded set ${\mathcal O}\subseteq {\mathbb R}^N,$ a continuous map
$\pi: {\overline{\mathcal O}}\to {\mathbb R}^{N-n}$ is a complement for the continuous map
$F: {\overline{\mathcal O}}\to {\mathbb R}^{n}$ if
the topological degree
$\mbox{\rm deg}((\pi,F),{\mathcal O},0)$ is defined and nonzero.
We also recall (see \cite{HoYo-61}) that a mapping $f$ of a space $X$ into a space $Y$ is
said to be \textit{inessential} if $f$ is homotopic to a constant; otherwise $f$ is
\textit{essential}.

\bigskip

\bigskip

\noindent
A more elementary version of Lemma \ref{th-3b.1} can be
given for the zero set of a vector field with range in ${\mathbb R}^{N-1}\,.$
In this case we can achieve our result by a direct use of the classical
Leray$-$Schauder continuation theorem \cite{LeSh-34},
instead of the more sophisticated tools
in \cite{FiMaPe-86}.
Namely, we have:

\begin{theorem}\label{th-3b.1a}
Let ${\mathcal R}:= \prod_{i=1}^N [a_i,b_i]$ be an $N-$dimensional rectangle
and let
$F = (F_1,\dots, F_{N-1}): {\mathcal R}\to {\mathbb R}^{N-1}$
be a continuous mapping such that, for each $i\in\{1,\dots,N-1\},$
$$F_i(x) < 0,\;\forall\, x\in {\mathcal R}^{\ell}_i\quad\mbox{and }\;
F_i(x) > 0,\;\forall\, x\in {\mathcal R}^{r}_i$$
or
$$F_i(x) > 0,\;\forall\, x\in {\mathcal R}^{\ell}_i\quad\mbox{and }\;
F_i(x) < 0,\;\forall\, x\in {\mathcal R}^{r}_i\,.$$
Then there exists a closed connected subset ${\mathcal Z}$ of
$$F^{-1}(0) = \{x\in {\mathcal R}:\, F_i(x) =0,\,\forall\, i=1,\dots,N-1\}$$
such that
$${\mathcal Z}\cap {\mathcal R}^{\ell}_N\not=\emptyset,\quad
{\mathcal Z}\cap {\mathcal R}^{r}_N\not=\emptyset.$$
\end{theorem}
\begin{proof}
We split $x = (x_1,\dots,x_{N-1},x_N)\in {\mathcal R}\subseteq {\mathbb R}^N$ as
$x = (y,\lambda)$ with
$$y= (x_1,\dots,x_{N-1})
\in {\mathcal M}:= \prod_{i=1}^{N-1} [a_i,b_i],\quad
\lambda = x_N\in [a_N,b_N]$$
and define
$$f = f(y,\lambda): {\mathcal M}\times [a_N,b_N]\to {\mathbb R}^{N-1}\,,\quad
f(y,\lambda):= F(x_1,\dots,x_{N-1},\lambda),$$
treating, in this manner, the variable $x_N =\lambda$ as a parameter for the
$(N-1)-$dimensional vector field
$$f_{\lambda}(\cdot) = f(\cdot,\lambda).$$
By the assumptions on $F$ we have
$$\mbox{\rm deg}(f_{\lambda}, {\mathcal M}^{^{^{\!\!\!\!\!{o}}}},0) = (-1)^d\not=0,
\quad\forall\, \lambda\in [a_N,b_N],$$ where $d$ is the number of
the components $i\in\{1,\dots,N-1\}$ such that $F_i(x) > 0$ for $x
\in {\mathcal R}^{\ell}_i$ (and also $F_i(x) < 0$ for $x\in
{\mathcal R}^{r}_i$). The Leray$-$Schauder continuation theorem
\cite[Th\'{e}or\`{e}me Fondamental]{LeSh-34} (see also
\cite{MaVi-89, Ma-97}) ensures the existence of a closed
connected set
$${\mathcal Z}\subseteq \{(y,\lambda)\in {\mathcal M}\times [a_N,b_N]\,:\, f(y,\lambda) =
0\in {\mathbb R}^{N-1}\}$$
whose projection onto the $\lambda-$component covers the interval $[a_N,b_N].$
By the above positions the thesis follows immediately.
\end{proof}
Theorem \ref{th-3b.1a} can be found also in \cite{KuSoTu-00} and
it was then applied in \cite{KuPoSoTu-05}.

\bigskip

\noindent
In the next lemma we take the unit cube
$I^N:=[0,1]^N$ as $N-$dimensional rectangle and choose
the interior point $P = (\tfrac 1 2,\tfrac 1 2,\dots, \tfrac 1 2).$

\begin{lemma}\label{th-3b.2}
Let $n\in \{1,\dots, N-1\}$ be fixed.
Assume that, for each $i\in\{1,\dots,n\},$ there exists a compact set
$${\mathcal S}_i\subseteq I^N$$
such that ${\mathcal S}_i$ cuts the arcs between $[x_i = 0]$
and $[x_i = 1]$ in $I^N.$ Then
there exists a connected subset ${\mathcal Z}$ of
$\,\displaystyle{\bigcap_{i=1}^n {\mathcal S}_i \not=\emptyset},$
whose dimension at each point is at least
$N-n.$ Moreover,
$$\mbox{\rm dim}({\mathcal Z}\cap\partial I^N)\geq N - n - 1$$
and
$$\pi: {\mathcal Z}\cap \partial I^N \to {\mathbb R}^{N-n}\setminus\{0\}$$
is essential (where $\pi$ is defined as in \eqref{eq-3b.1}).
\end{lemma}
\begin{proof}
For any fixed index $i^*\in\{1,\dots,n\}$ we define the {\textit{tunnel set}}
$$T_{i^*}\,:= \prod_{i=1}^{i^*-1} [0, 1]\times {\mathbb R}\times
\prod_{i=i^*+1}^N [0,1].$$
It is immediate to check that ${\mathcal S}_{i^*}$ cuts the arcs between
$[x_{i^*} = 0]$
and $[x_{i^*} = 1]$ in $T_{i^*}\,.$
\\
By Lemma \ref{lem-ABC.5} there exists a continuous function
$f_{i^*}\,: T_{i^*}\to {\mathbb R}$ such that
$$f_{i^*}(x) \leq 0\,,\;\forall\, x\in T_{i^*}\,\;\mbox{with } x_{i^*}\leq 0\quad
\mbox{and } \;
f_{i^*}(x) \geq 0\,,\;\forall\, x\in T_{i^*}\,\;\mbox{with } x_{i^*}\geq 1\,.$$
Moreover,
$${\mathcal S}_{i^*}\,= \{ x\in T_{i^*}\,: \, f_{i^*}(x) = 0\}.$$
By this latter property and the fact that ${\mathcal S}_{i^*}\subseteq I^N$
it follows that
$$f_{i^*}(x) < 0\,,\;\forall\, x\in T_{i^*}\,\;\mbox{with } x_{i^*}< 0\quad
\mbox{and } \;
f_{i^*}(x) > 0\,,\;\forall\, x\in T_{i^*}\,\;\mbox{with } x_{i^*} >  1\,.$$
Now we define, for $x = (x_1,\dots,x_{i^*-1},x_{i^*},x_{i^*+1},\dots, x_N)\in {\mathbb R}^N,$
the continuous function
$$F_{i^*}(x):= f_{i^*}\bigl(\,\eta_{[0,1]}(x_1),\dots,
\eta_{[0,1]}(x_{i^* -1}), x_{i^*},
\eta_{[0,1]}(x_{i^* +1}),\dots,
\eta_{[0,1]}(x_{N})\,\bigr ),$$
where
$\eta_{[0,1]}$
is the projection of ${\mathbb R}$ onto the interval $[0,1]$
defined as in \eqref{eq-2b.e}. As a consequence of the above positions we find that
$$F_{i^*}(x) < 0,\;\forall\, x\in{\mathbb R}^N\,: \, x_{i^*} < 0\quad\mbox{and }\;
F_{i^*}(x) > 0,\;\forall\, x\in{\mathbb R}^N\,: \, x_{i^*} >  1\,.$$
We are ready to apply Theorem \ref{th-3b.1} to the map
$F = (F_1,\dots,F_n)$ restricted to the rectangle
$${\mathcal R}:= \prod_{i=1}^{n}[-1,1] \times \prod_{i=n+1}^{N}[0,1].$$
Clearly,
$$\bigl(F|_{\mathcal R}\bigr)^{-1}(0) = \bigcap_{i=1}^n {\mathcal S}_i\subseteq I^N$$
and the thesis is achieved.
\end{proof}

\begin{remark}\label{rem-3b.2}
Both in Theorem \ref{th-3b.1} and in Lemma \ref{th-3b.2} the fact that we have privileged
the first $n$ components is purely conventional. It is evident that the results are still true
for any finite sequence of indexes $i_1 < i_2 < \dots < i_n$ in $\{1,\dots,N\}.$
Moreover, Lemma \ref{th-3b.2} is invariant under homeomorphisms in a sense that is described
in Theorem \ref{th-3b.2a} below. The same observation applies systematically to all the other results
(preceding and subsequent) in which some directions are conventionally chosen.
\hfill$\lhd$
\end{remark}
In view of the next result we recall the concept of generalized rectangle ${\widehat{X}}:= ({X},h)$ given in
Definition \ref{def-2b.1}, where $h: I^N\to X\subseteq Z$ is a homeomorphism
of the unit cube $I^N = [0,1]^N$
onto its image $X.$

\begin{theorem}\label{th-3b.2a}
Let ${\widehat{X}}:= ({X},h)$ be a generalized rectangle of a metric space $Z.$
Let a finite sequence of $n$ indexes $i_1 < i_2 < \dots < i_n$ ($n\geq 1$)
be fixed in $\{1,\dots,N\}.$
Assume that, for each $j\in\{i_1,\dots,i_n\},$ there exists a compact set
$${\mathcal S}_j\subseteq {X}$$
such that ${\mathcal S}_j$ cuts the arcs between ${X}_{j}^{\ell}$
and ${X}_{j}^{r}$ in $X.$
Then there exists a compact connected subset ${\mathcal Z}$ of
$\,\displaystyle{\bigcap_{k=1}^n {\mathcal S}_{i_k}\not=\emptyset},$
whose dimension at each point is at least
$N-n.$ Moreover,
$$\mbox{\rm dim}({\mathcal Z}\cap\vartheta X)\geq N - n - 1$$
and
$$\pi: h^{-1}({\mathcal Z})\cap \partial I^N \to {\mathbb R}^{N-n}\setminus\{0\}$$
is essential (where $\pi$ is defined as in \eqref{eq-3b.1}).
\end{theorem}
\begin{proof}
The result easy follows by moving to the setting of Lemma \ref{th-3b.2}
through the homeomorphism $h^{-1}$ and repeating the arguments employed therein.
\end{proof}

\noindent
We end this section by presenting a result (Corollary \ref{cor-3b.1}) which plays a crucial role
in the subsequent proofs.
It concerns the case $n=N-1$ and could be obtained by
suitably adapting the arguments employed in Lemma \ref{th-3b.2}.
However, due to its significance for our applications we prefer to provide a detailed proof
using Theorem \ref{th-3b.1a}
(which requires only the knowledge of the Leray$-$Schauder principle
and therefore, in some sense, is more elementary).
Corollary \ref{cor-3b.1} extends to an arbitrary dimension some results
in \cite[Appendix]{ReZa-00} which were proved only for $N=2$ using \cite{Sa-80}.

\begin{corollary}\label{cor-3b.1}
Let ${\widehat{X}}:= ({X},h)$ be a generalized rectangle in a metric space $Z.$
Let $k\in \{1,\dots,N\}$ be fixed.
Assume that, for each $j\in\{1,\dots,N\}$ with $j\not=k,$
there exists a compact set
$${\mathcal S}_j\subseteq {X}$$
such that ${\mathcal S}_j$ cuts the arcs between ${X}_{j}^{\ell}$
and ${X}_{j}^{r}$ in $X.$
Then there exists a compact connected subset ${\mathcal C}$ of
$\,\displaystyle{\bigcap_{i=1\atop i\not=k}^N {\mathcal S}_{i}\not=\emptyset},$
such that
$${\mathcal C}\cap {X}_{k}^{\ell}\not=\emptyset,\quad
{\mathcal C}\cap {X}_{k}^{r}\not=\emptyset.$$
\end{corollary}
\begin{proof}
Without loss of generality (if necessary, by a permutation of the
coordinates), we assume $k=N.$ In this manner,
using the homeomorphism $h^{-1}: Z\supseteq X= h(I^N)\to I^N,$
we can confine ourselves to the
following situation:
\\
For each $j\in\{1,\dots,N-1\},$ there exists a compact set
$${\mathcal S}\,'_j \,:= h^{-1}({\mathcal S}_j)\subseteq I^N$$
that cuts the arcs between $[x_i=0]$ and $[x_i=1]$
in $I^N\,.$

Proceeding as in the proof of Lemma \ref{th-3b.2},
for any fixed index $i^*\in\{1,\dots,N-1\}$ we define the {\textit{tunnel set}}
$$T_{i^*}\,:= \prod_{i=1}^{i^*-1} [0,1]\times {\mathbb R}\times
\prod_{i=i^*+1}^N [0, 1]$$
and find that ${\mathcal S}\,'_{i^*}$ cuts the arcs between
$[x_{i^*} = 0]$
and $[x_{i^*} = 1]$ in $T_{i^*}\,.$
\\
Hence, by Lemma \ref{lem-ABC.5} there exists a continuous function
$f_{i^*}\,: T_{i^*}\to {\mathbb R}$ such that
$$f_{i^*}(x) \leq 0\,,\;\forall\, x\in T_{i^*}\,\;\mbox{with } x_{i^*}\leq 0\quad
\mbox{and } \;
f_{i^*}(x) \geq 0\,,\;\forall\, x\in T_{i^*}\,\;\mbox{with } x_{i^*}\geq  1\,.$$
Moreover,
$${\mathcal S}\,'_{i^*}\,= \{ x\in T_{i^*}\,: \, f_{i^*}(x) = 0\},$$
as well as
$$f_{i^*}(x) < 0\,,\;\forall\, x\in T_{i^*}\,\;\mbox{with } x_{i^*}< 0\quad
\mbox{and } \;
f_{i^*}(x) > 0\,,\;\forall\, x\in T_{i^*}\,\;\mbox{with } x_{i^*} >  1\,.$$
We define, for $x = (x_1,\dots,x_{i^*-1},x_{i^*},x_{i^*+1},\dots, x_N)\in {\mathbb R}^N,$
the continuous function
$$F_{i^*}(x):= f_{i^*}\bigl(\,\eta_{[0,1]}(x_1),\dots,
\eta_{[0, 1]}(x_{i^* -1}), x_{i^*},
\eta_{[0, 1]}(x_{i^* +1}),\dots,
\eta_{[0, 1]}(x_{N})\,\bigr ),$$
where
$\eta_{[0, 1]}$
is the projection of ${\mathbb R}$ onto the interval $[0, 1]$
defined as in \eqref{eq-2b.e}. Then we have
$$F_{i^*}(x) < 0,\;\forall\, x\in{\mathbb R}^N\,: \, x_{i^*} < 0\quad\mbox{and }\;
F_{i^*}(x) > 0,\;\forall\, x\in{\mathbb R}^N\,: \, x_{i^*} >  1\,.$$
Now we consider the map
$F = (F_1,\dots,F_{N-1})$ restricted to the rectangle
$${\mathcal R}:= \prod_{i=1}^{N-1}[-\varepsilon,1+\varepsilon]
\times [0,1],$$
for any fixed $\varepsilon > 0.$
Since
$$\bigl(F|_{\mathcal R}\bigr)^{-1}(0) = \bigcap_{i=1}^{N-1} {\mathcal S}\,'_i\subseteq I^N,$$
the thesis follows by Theorem \ref{th-3b.1a}. In fact, we can define the set
$${\mathcal C}:= h({\mathcal Z}),$$
where ${\mathcal Z} \subseteq \bigl(F|_{\mathcal R}\bigr)^{-1}(0)$
comes from the statement of Theorem \ref{th-3b.1a}.
\end{proof}

\section{Periodic points and chaotic dynamics for maps which expand the paths}\label{sec-3c}
We provide now an extension to $N-$dimensional spaces of some results previously
obtained in \cite{PaZa-04a, PaZa-04b} for the planar case.
As in \cite{PaZa-04a, PaZa-04b} we are interested in the study
of maps which expand the arcs along a certain direction. To this aim, we
reconsider Definition \ref{def-2b.1} in order to focus our attention on a
generalized $N-$dimensional rectangle in which we have fixed (once for all)
the left and right sides. In the applications, these opposite sides
give an orientation (in a rough sense) of the generalized rectangle
and they will be related to the expansive direction.

\begin{definition}\label{def-3c.1}
Let $Z$ be a metric space and let
${\widehat{X}}:= ({X},h)$ be a
generalized $N-$dimensional rectangle of $Z.$
We set
$$X_{\ell}:= h([x_N=0]),\quad X_{r}:= h([x_N=1])$$
and
$$X^-:= X_{\ell}\cup X_{r}\,.$$
The pair
$${\widetilde{X}}:= (X, X^-)$$
is called an {\textit{oriented $N-$dimensional rectangle}}
(or, simply, an oriented rectangle)
\textit{of $Z$}. For simplicity,
the reference to the ambient space $Z$ will be omitted
when no possibility of confusion may occur.
\end{definition}

\begin{remark}\label{rem-3c.1}
First of all we observe that, instead of the unit cube $[0,1]^N,$ we could have chosen
in the above definition any $N-$dimensional rectangle.
In this case the sides $X_{\ell}$ and $X_{r}$
would be defined (in a obvious manner) accordingly.

A comparison between Definition \ref{def-2b.1} and Definition \ref{def-3c.1}
shows that an oriented rectangle is just a generalized rectangle in which
we have privileged the two subsets of its contour which correspond to the
opposite faces for some fixed component (namely, the $x_N-$component).
The choice of the $N-$th component is purely conventional.
For example, in some other papers (see \cite{GiRo-03, PiZa-05, ZgGi-04}),
the first component was selected.
Clearly, there is no substantial difference as the homeomorphism $h$
could be composed with a permutation matrix (yielding to a new
homeomorphism with the same image set). From this point of view, our
definition fits to the one of $h-$set of $(1,N-1)-$type,
given by Zgliczy\'{n}ski and Gidea in \cite{ZgGi-04}
for a subset of ${\mathbb R}^N$ which is obtained as the counterimage
of the unit cube through a homeomorphism of ${\mathbb R}^N$ onto itself.
A similar concept is also considered by Gidea and Robinson in \cite{GiRo-03}:
they call this object a $(1,N-1)-$window and it is defined as
a homeomorphic copy of the unit cube $I^N$ of $\mathbb R^N$
through a homeomorphism whose domain is an open neighborhood of $I^N.$
\hfill$\lhd$\\
\end{remark}

The next definition introduces the concept of stretching along the paths
(already considered in \cite{PaZa-04a, PaZa-04b, PiZa-05}) for maps
between oriented rectangles.

\begin{definition}\label{def-3c.2}
Let $Z$ be a metric space and let
${\widetilde{X}}:= (X, X^-)$ and
${\widetilde{Y}}:= (Y, Y^-)$
be {oriented $N-$dimensional rectangles} of $Z.$
Let $\psi: Z\supseteq D_{\psi}\to Z$ be a map (not necessarily continuous
on its whole domain $D_{\psi}$) and let
$${\mathcal D}\subseteq X\cap D_{\psi}\,.$$
We say that {\textit{$({\mathcal D},\psi)$ stretches ${\widetilde{X}}$ to
${\widetilde{Y}}$ along the paths}} and write
$$({\mathcal D},\psi): {\widetilde{X}}\stretchx {\widetilde{Y}}$$
if there exists a compact set ${\mathcal K}\subseteq {\mathcal D}$
such that $\psi$ is continuous on ${\mathcal K}$ and
for every path $\gamma$ with
$${\overline{\gamma}}\subseteq X\quad\mbox{and }\;
{\overline{\gamma}}\cap X_{\ell}\not=\emptyset,\;\;
{\overline{\gamma}}\cap X_{r}\not=\emptyset,$$
there is a sub-path $\sigma$ of $\gamma$ such that
$${\overline{\sigma}}\subseteq {\mathcal K}\quad\mbox{and }\;
\psi({\overline{\sigma}})\subseteq Y,\quad \mbox{with }\;
\psi({\overline{\sigma}})\cap Y_{\ell}\not=\emptyset,\;\;
\psi({\overline{\sigma}})\cap Y_{r}\not=\emptyset.$$
We also write
$$({\mathcal D},{\mathcal K},\psi): {\widetilde{X}}\stretchx {\widetilde{Y}}$$
when we wish to put in evidence the role of the set ${\mathcal K}.$
In some applications, we take ${\mathcal K}\subseteq {\mathcal D}$ such that
$\psi({\mathcal K})\subseteq Y.$ In this case, the condition $\psi({\overline{\sigma}})\subseteq Y$
is automatically satisfied.
\end{definition}

\bigskip

\begin{remark}\label{rem-3c.2}
Let
${\widetilde{X}}= (X, X^-)$ and
${\widetilde{Y}}= (Y, Y^-)$
be {oriented $N-$dimensional rectangles} of a metric space $Z$
and assume that
$$({\mathcal D},\psi): {\widetilde{X}}\stretchx {\widetilde{Y}}$$
for some $\psi: Z\supseteq D_{\psi}\to Z$
and ${\mathcal D}\subseteq X\cap D_{\psi}\,.$ From the above definition it turns out that
$$({\mathcal D}\,',\psi): {\widetilde{X}}\stretchx {\widetilde{Y}},
\quad
\forall\, {\mathcal D}\,'\,:\; \psi^{-1}(Y)\cap {\mathcal D}\subseteq {\mathcal D}\,'\subseteq X\cap D_{\psi}\,.$$
We also note that if ${\mathcal D}$ is closed and $\psi$ is continuous on ${\mathcal D},$ we
can take ${\mathcal K} = {\mathcal D}$ in the definition.

Clearly, there are situations where there is no need
to invoke the set ${\mathcal K}$ because the knowledge of ${\mathcal D}$
gives the required information. A particular case in which it is not necessary to specify such a compact ${\mathcal K}$ occurs when $X\subseteq D_{\psi}$:
indeed, an easy criterion to verify the stretching condition in this particular context is checking that $\psi(X)\subseteq Y$
and $\psi(X_{\ell})\subseteq Y_{\ell}$ as well as $\psi(X_{r})\subseteq Y_{r}\,,$ or
$\psi(X_{\ell})\subseteq Y_{r}$ as well as $\psi(X_{r})\subseteq Y_{\ell}\,.$

On the other hand, in some cases, it may be useful to emphasize the existence of a special set
${\mathcal K}.$ For instance, we could be interested in examples where
$({\mathcal D},{\mathcal K}_i,\psi): {\widetilde{X}}\stretchx {\widetilde{Y}}$
for different (even disjoint) sets  ${\mathcal K}_i$'s  (see, e.g., Theorem \ref{th-3c.3})
and also in situations in which either $\psi$ is not defined on the whole set $X$ or
$\psi$ is defined on $X$ but $\psi(X)\not\subseteq Y.$
\hfill$\lhd$\\
\end{remark}

As a consequence of Corollary \ref{cor-3b.1} we obtain the following result
which extends to $N-$dimensional rectangles a fixed point theorem
(see \cite[Th. 3.1]{PaZa-04b}), originally proved for $N=2.$

\begin{theorem}\label{th-3c.1}
Let ${\widetilde{X}}:= (X, X^-)$ be an oriented $N-$dimensional rectangle
of a metric space $Z$ and let $\psi: Z\supseteq D_{\psi}\to Z$
and ${\mathcal D}\subseteq X\cap D_{\psi}$ be such that
\begin{equation}\label{eq-3c.DK}
({\mathcal D},{\mathcal K},\psi): {\widetilde{X}}\stretchx {\widetilde{X}},
\end{equation}
for some compact set ${\mathcal K}\subseteq {\mathcal D}.$
Then there exists ${\widetilde{w}}\in {\mathcal K}$ such that
$${\psi}({\widetilde{w}}\,) = {\widetilde{w}}.$$
\end{theorem}
\begin{proof}
Let $h: I^N=[0,1]^N\to h(I^N)=X\subseteq Z$ be a homeomorphism such that $X_{\ell}= h([x_N=0])$
and $X_{r}= h([x_N=1])$ and consider the compact set of $I^N$
$${\mathcal W}:= h^{-1}({\mathcal K}\cap \psi^{-1}(X))$$
and the continuous mapping $\phi = (\phi_1,\dots,\phi_N): {\mathcal W}\to I^N$ defined by
$$\phi(x):= h^{-1}\bigl(\psi( h(x))\bigr ),\quad\forall\, x\in {\mathcal W}.$$
By the Tietze$-$Urysohn theorem \cite[p.87]{En-78} there exists a continuous map
$$\varphi = (\varphi_1,\dots,\varphi_N): I^N\to I^N\,,\quad \varphi|_{\mathcal W}= \phi.$$
Let use define, for every $i = 1,\dots, N-1,$ the closed sets
$${\mathcal S}_i:= \{x=(x_1,\dots,x_{N-1},x_N)\in I^N\,: \, x_i = \varphi_i(x)\}\subseteq I^N.$$
Since $\varphi(I^N)\subseteq I^N,$ by the continuity of the $\varphi_i$'s, it is straightforward
to check that ${\mathcal S}_i$ cuts the arcs between $[x_i = 0]$ and $[x_i=1]$ in $I^N$
(for each $i=1,\dots,N-1$). Indeed, if $\gamma: [0,1]\to I^N$ is a path with $\gamma_i(0) = 0$ and
$\gamma_i(1)=1,$ then, for the auxiliary function
$g: [0,1]\ni t\mapsto \gamma_i(t) - \varphi_i(\gamma(t)),$ we have $g(0) \leq 0 \leq g(1)$
and therefore there exists $s\in [0,1]$ such that
$\gamma_i(s) = \varphi_i(\gamma(s))$ (Bolzano theorem), that is ${\overline{\gamma}}\cap {\mathcal S}_i\not=\emptyset.$
Thus the cutting property is proved.

\medskip

Now Corollary \ref{cor-3b.1} guarantees the existence of a continuum
\begin{equation}\label{eq-3c.CS}
{\mathcal C}\subseteq \bigcap_{i=1}^{N-1}{\mathcal S}_i
\end{equation}
such that
$${\mathcal C}\cap [x_N=0]\not=\emptyset,\quad {\mathcal C}\cap [x_N=1]\not=\emptyset.$$
By Lemma \ref{lem-ABC.6} we have that, for every $\varepsilon > 0,$ there exists a path
$\gamma_{\varepsilon}: [0,1]\to I^N$ such that
$$\gamma_{\varepsilon}(0) \in [x_N=0],\quad \gamma_{\varepsilon}(1) \in [x_N=1]\quad
\mbox{and }\; \gamma_{\varepsilon}(t)\in B({\mathcal C},\varepsilon)\cap I^N,\;\forall\, t\in [0,1].$$
By \eqref{eq-3c.DK} and the definition of ${\mathcal W}$ and $\phi,$ there exists a sub-path
$\sigma_{\varepsilon}$ of $\gamma_{\varepsilon}$ such that
$${\overline{\sigma_{\varepsilon}}}\subseteq {\mathcal W}\quad\mbox{and }\;
\phi({\overline{\sigma_{\varepsilon}}})\subseteq I^N,\quad \mbox{with }\;
\phi({\overline{\sigma_{\varepsilon}}})\cap [x_N=0]\not=\emptyset,\;\;
\phi({\overline{\sigma_{\varepsilon}}})\cap [x_N=1]\not=\emptyset.$$

The Bolzano theorem applied this time to the continuous mapping $x\mapsto x_N - \varphi_N(x)$
on ${\overline{\sigma_{\varepsilon}}}$ implies the existence of a point
$${\tilde{x}\,}^{\varepsilon}= ({\tilde{x}\,}^{\varepsilon}_1,\dots,
{\tilde{x}\,}^{\varepsilon}_N)\in {\overline{\sigma_{\varepsilon}}} \subseteq {\mathcal W}$$
such that
$${\tilde{x}\,}^{\varepsilon}_N = \varphi_N({\tilde{x}\,}^{\varepsilon}).$$
Taking $\varepsilon = \tfrac 1 n$ and letting $n\to \infty,$ by a standard compactness argument
we find a point
$${\tilde{x}} = ({\tilde{x}}_1,\dots,{\tilde{x}}_N)\in {\mathcal C} \cap {\mathcal W}$$
such that
$${\tilde{x}}_N = \varphi_N({\tilde{x}}).$$
By \eqref{eq-3c.CS}, recalling also the definition of the ${\mathcal S}_i$'s, we find
$${\tilde{x}} = \varphi({\tilde{x}})\in {\mathcal W}.$$
Then, since $\varphi|_{\mathcal W}= \phi,$ by the relation
$$h( \phi(x) ) = \psi( h(x) ),\quad\forall\, x\in {\mathcal W},$$
we have that $h({\tilde{x}}) = \psi( h({\tilde{x}}) )\in h({\mathcal W})$ and therefore
$${\widetilde{w}}:= h({\tilde{x}})\in {\mathcal K}\cap \psi^{-1}(X)$$
is the desired fixed point for $\psi.$
\end{proof}

\bigskip

Having proved Theorem \ref{th-3b.2a} and Theorem \ref{th-3c.1}, we have now available the tools for
extending to any dimension the results about periodic points and chaotic dynamics previously obtained
for the two$-$dimensional case in \cite{PaZa-04a, PaZa-04b}. For sake of conciseness we focus our attention only on
some of them (selected from \cite{PaZa-04a, PaZa-04b}), that we present below in the more general setting.

\begin{theorem}\label{th-3c.2}
Assume there is a double sequence of oriented $N-$dimensional rectangles $({\widetilde{X}_k})_{k\in {\mathbb Z}}$
(with ${\widetilde{X}_k} = (X_k,X^-_k)$) of a metric space $Z$ and a sequence $(({\mathcal D}_k,\psi_k))_{k\in {\mathbb Z}}\,,$
with ${\mathcal D}_k \subseteq X_k\,,$ such that
$$({\mathcal D}_k,\psi_k): {\widetilde{X}_k}  \stretchx {\widetilde{X}_{k+1}}\,,\quad\forall\, k\in {\mathbb Z}.$$
Let us denote by
$X^{k}_{\ell}$ and $X^{k}_{r}$ the two parts of $X_k^-\,.$
Then the following conclusions hold:
\begin{itemize}
\item[$(a_1)\;$] There is a sequence $(w_k)_{k\in {\mathbb Z}}$ such that $w_k\in {\mathcal D}_k$
and $\psi_k(w_k) = w_{k+1}\,,$ for all ${k\in {\mathbb Z}}\,;$
\item[$(a_2)\;$] For each $j\in {\mathbb Z}$ there exists a compact connected set ${\mathcal C}_j\subseteq {\mathcal D}_j$
which cuts the arcs between $X^{j}_{\ell}$ and $X^{j}_{r}$ in $X_j$ and such that, for every $w\in {\mathcal C}_j\,,$
there is a sequence $(y_i)_{i\geq j}$ with $y_j = w$ and
$$y_i \in {\mathcal D}_i\,,\quad \psi_{i}(y_i) = y_{i+1}\,,\;\forall\, i\geq j.$$
The dimension of ${\mathcal C}_j$ at each point is at least $N-1.$
Moreover, $\mbox{\rm dim}({\mathcal C}_j\cap \vartheta X_j)\geq N - 2$ and
$\pi: h^{-1}({\mathcal C}_j)\cap \partial I^N \to {\mathbb R}^{N-1}\setminus\{0\}$
is essential (where $\pi$ is defined as in \eqref{eq-3b.1} for $p_i=\tfrac 1 2,\,\forall\, i$);
\item[$(a_3)\;$] If there are integers $k$ and $l,$ with $k < l,$ such that ${\widetilde{X}_k} = {\widetilde{X}_l}\,,$
then there exists a finite sequence $(z_i)_{k\leq i\leq l}\,,$ with $z_i\in {\mathcal D}_i$ and
$\psi_i(z_i) = z_{i+1}$ for each $i=k,\dots,l-1,$ such that $z_l = z_k\,,$ that is,
$z_k$ is a fixed point of $\psi_{l-1}\circ\dots\circ\psi_{k}\,.$
\end{itemize}
\end{theorem}
\begin{proof}
We prove the conclusions of the theorem in the reverse order. So,
let's start with the verification of $(a_3).$ By the assumptions
and the definition of the ``stretching along the paths'' property, it
is easy to check that
\begin{equation}\label{eq-3c.comp}
({\mathcal D},\psi_{l-1}\circ\dots\circ\psi_k): {\widetilde{X}_k}
\stretchx {\widetilde{X}_{l}},
\end{equation}
where
$$\mathcal D:=\{z\in \mathcal D_k:\psi_{j}\circ\dots\circ\psi_k(z)\in\mathcal D_{j+1},\,\forall j=k,\dots,l-1\}.$$
With the positions $\widetilde X={\widetilde{X}_k} =
{\widetilde{X}_l}$ and $\psi=\psi_{l-1}\circ\dots\circ\psi_k,$
we read condition \eqref{eq-3c.comp} as $({\mathcal D},\psi): {\widetilde{X}}
\stretchx {\widetilde{X}}$ and therefore the thesis follows immediately by Theorem \ref{th-3c.1}.
More precisely, if we like to put in evidence the role of the compact sets ${\mathcal K}_k$'s,
for $({\mathcal D}_k,{\mathcal K}_k,\psi_k): {\widetilde{X}_k}  \stretchx {\widetilde{X}_{k+1}},$
we have that
\begin{equation*}
({\mathcal D},{\mathcal K},\psi): {\widetilde{X}}\stretchx {\widetilde{X}},
\end{equation*}
where we have set
$${\mathcal K}:=\{z\in \mathcal K_k:\psi_{j}\circ\dots\circ\psi_k(z)\in\mathcal K_{j+1},\,\forall j=k,\dots,l-1\}.$$
As regards $(a_2),$ without loss of generality, we can assume
$j=0.$ Recall that by Definition \ref{def-3c.2}, since $({\mathcal
D}_k,\psi_k): {\widetilde{X}_k}  \stretchx
{\widetilde{X}_{k+1}},\,\forall\, k\in {\mathbb Z},$ it follows
that for any $k$ there exists a compact set $\mathcal K_k\subseteq
X_k$ such that $\psi_k$ is continuous on ${\mathcal K_k}$ and for
every path $\gamma$ with ${\overline{\gamma}}\subseteq X_k$ and
${\overline{\gamma}}\cap X_{\ell}^k\not=\emptyset,\; {\overline{\gamma}}\cap
X_{r}^k\not=\emptyset,$ there is a sub-path $\sigma$ of $\gamma$
such that ${\overline{\sigma}}\subseteq {\mathcal K}_k$ and
$\psi({\overline{\sigma}})\subseteq X_{k+1},$ with
$\psi({\overline{\sigma}})\cap X_{\ell}^{k+1}\not=\emptyset,\;
\psi({\overline{\sigma}})\cap X_{r}^{k+1}\not=\emptyset.$ Let us define
the closed set
\begin{equation}
\mathcal S:=\{z\in {\mathcal K}_0: \psi_j\circ\dots\circ\psi_0(z)\in
{\mathcal K}_{j+1},\,\forall j\geq 0\}
\end{equation}
and fix a path $\gamma_0$ such that $\overline{{\gamma}_0}\subseteq X_0$
and $\overline{{\gamma}_0}\cap X_{\ell}^0\not=\emptyset,\;
\overline{{\gamma}_0}\cap X_{r}^0\not=\emptyset.$ Then, since
$({\mathcal D}_0,\psi_0): {\widetilde{X}_0} \stretchx
{\widetilde{X}_{1}},$ there exists a sub-path $\gamma_1$ of
$\gamma_0$ with $\overline{{\gamma}_1}\subseteq \mathcal K_0\subseteq
X_0$ such that $\psi_0({\overline{{\gamma}_1}})\subseteq X_1$ and
$\psi_0({\overline{{\gamma}_1}})\cap X_{\ell}^1\not=\emptyset,\;
\psi_0({\overline{{\gamma}_1}})\cap X_{r}^1\not=\emptyset.$ Similarly,
there exists a sub-path $\sigma_2$ of
$\sigma_1:=\psi_0({\gamma}_1)$ with $\overline{{\sigma}_2}\subseteq
\mathcal K_1\subseteq\mathcal D_1$ and such that
$\psi_1({\overline{{\sigma}_2}})\subseteq X_2,\;
\psi_1({\overline{{\sigma}_2}})\cap X_{\ell}^2\not=\emptyset,\;
\psi_1({\overline{{\sigma}_2}})\cap X_{r}^2\not=\emptyset.$ Defining
$$\Gamma_2:=\{x\in\overline{{\gamma}_1}:\psi_0(x)\in\overline{{\sigma}_2}\}\subseteq \{z\in {\mathcal K}_0:\psi_0(z)\in
{\mathcal K}_{1}\}
$$
and proceeding by induction, we can find a decreasing sequence of
nonempty compact sets
$$\Gamma_0:=\overline{{\gamma}_0}\supseteq\Gamma_1:=\overline{{\gamma}_1}
\supseteq\Gamma_2\supseteq\dots\supseteq\Gamma_n\supseteq\Gamma_{n+1}\supseteq\dots$$
such that
$\psi_j\circ\dots\circ\psi_0(\Gamma_{j+1})\subseteq
X_{j+1},\, \psi_j\circ\dots\circ\psi_0(\Gamma_{j+1})\cap
X_{\ell}^{j+1}\not=\emptyset,\;
\psi_j\circ\dots\circ\psi_0(\Gamma_{j+1})\cap
X_{r}^{j+1}\not=\emptyset,$ for $j\geq 0.$
Moreover, for every $i\geq 1,$ we have that
$$\Gamma_{i+1}\subseteq \{z\in {\mathcal K}_0:
\psi_{j-1}\circ\dots\circ\psi_0(z)\in
{\mathcal K}_{j},\,\forall \,j: 1\leq j\leq i\}.
$$
It is easy to see that
${\displaystyle{\cap_{j=0}^{+\infty}\,\Gamma_j\ne\emptyset}}$ and for any
${\displaystyle{z\in\cap_{j=0}^{+\infty}\,\Gamma_j}}$ it holds that
$\psi_n\circ\dots\circ\psi_0(z)\in\mathcal K_{n+1},\,\forall
n\in\mathbb N.$ In this way we have shown that any path
$\gamma_0,$ with $\overline{{\gamma}_0}$ joining the two sides of
$X_0^{-},$ intersects $\mathcal S,$ i.e. $\mathcal S$ cuts the
arcs between $X^{0}_{\ell}$ and $X^{0}_{r}$ in $X_0.$ Obviously
any point belonging to the intersection of $\overline{{\gamma}_0}$ with
$\mathcal S$ generates a sequence with the properties required in
$(a_2).$
The existence of the connected compact set ${\mathcal C}_0
\subseteq {\mathcal D}_0$ as in $(a_2)$ follows from Theorem \ref{th-3b.2a}, setting $n=1$ and $i_1=0.$ \\
Conclusion $(a_1)$ follows now from $(a_2)$ by a standard diagonal
argument already employed in previous works (see, e.g., \cite[Proposition 5]{KeKoYo-01},
\cite[Theorem 2.2]{PaZa-04a}).
\end{proof}

\begin{remark}\label{rem-3c.3}
An apparently more general version of Theorem \ref{th-3c.2} can be obtained by assuming
the $X_k$'s to be contained in possibly different metric spaces $Z_k$'s.
\\
If, at any step $k\in{\mathbb Z},$ we have the further information that $({\mathcal D}_k,{\mathcal K}_k,\psi_k):
{\widetilde{X}_{k}} \stretchx {\widetilde{X}_{k+1}}\,,$ then, in each of the corresponding conclusions
$(a_1),$ $(a_2),$ $(a_3)$ we
can be more precise and add that $w_k\in {\mathcal K}_k\,,\,y_k\in {\mathcal K}_k\,,$ or $z_k\in {\mathcal K}_k\,,$
respectively.
\hfill$\lhd$\\
\end{remark}

We end this paper with a few consequences of Theorem \ref{th-3c.2}.

\bigskip
Our applications deal with discrete dynamical systems exhibiting a chaotic behavior.
Due to the many different definitions of chaos available in the literature, we state in a
precise manner the one we use. The same concept of chaos has been already considered
in \cite{PaZa-04a, PaZa-04b, PaZa->, PiZa-05} as well as
in previous works by other authors (see, for instance, \cite{Zg-96}).

\noindent
\begin{definition}\label{def-3c.3}
Let $Z$ be a metric space,  $\psi: Z\supseteq D_{\psi} \to Z$ be a map
and let ${\mathcal D}\subseteq D_{\psi}\,.$ Assume also that $m\geq 2$ is an integer.
We say that
\textit{$\psi$ induces chaotic dynamics on $m$ symbols in the set ${\mathcal D}$}
if there exist $m$ nonempty pairwise disjoint compact sets
$${\mathcal K}_0\,,\; {\mathcal K}_1\,,\dots, {\mathcal K}_{m-1}\,\subseteq {\mathcal D},$$
such that, for each two$-$sided sequence $(s_i)_{i\in {\mathbb Z}} \in  \{0,\dots,m-1\}^{\mathbb Z},$
there exists a corresponding sequence $(w_i)_{i\in {\mathbb Z}}\in {\mathcal D}^{\mathbb Z}$
such that
\begin{equation}\label{eq-3c.ct1}
w_i \,\in\, {\mathcal K}_{s_i}\;\;\mbox{ and }\;\, w_{i+1} = \psi(w_i),\;\; \forall\, i\in {\mathbb Z}
\end{equation}
and, whenever $(s_i)_{i\in {\mathbb Z}}$ is a $k-$periodic sequence (that is,
$s_{i+k} = s_i\,,\forall i\in {\mathbb Z}$) for some $k\geq 1,$
there exists a $k-$periodic sequence $(w_i)_{i\in {\mathbb Z}}\in {\mathcal D}^{\mathbb Z}$
satisfying \eqref{eq-3c.ct1}. When we want to stress the role of the
${\mathcal K}_j$'s, we also say that
\textit{$\psi$ induces chaotic dynamics on $m$ symbols in the set ${\mathcal D}$ relatively to
$({\mathcal K}_0,\dots,{\mathcal K}_{m-1})$}.
\end{definition}

\begin{remark}\label{rem-3c.4}
We recall that the property expressed in \eqref{eq-3c.ct1} corresponds
(in the case of two symbols) to the definition
of {\textit{chaos in the sense of coin$-$tossing}}
considered  by Kirchgraber and Stoffer in \cite{KiSt-89}. The same kind of chaotic behavior
is also obtained by Kennedy, Ko{\c{c}}ak e Yorke in \cite[Proposition 5]{KeKoYo-01}.
As a further addition with respect to \cite{KeKoYo-01} and \cite{KiSt-89},
our definition takes account also of the presence of periodic itineraries
generated by periodic points.
\hfill$\lhd$\\
\end{remark}

\begin{theorem}\label{th-3c.3}
Assume there is an oriented $N-$dimensional rectangle ${\widetilde{X}} = (X,X^-)$
of a metric space $Z$ and a map $\psi: Z\supseteq D_{\psi}\to Z.$
Let ${\mathcal D}\subseteq X \cap D_{\psi}$ and suppose there exist
$m\geq 2$ nonempty and pairwise disjoint compact sets
${\mathcal K}_0\,,{\mathcal K}_1\,,\dots, {\mathcal K}_{m-1}\,\subseteq {\mathcal D}$ such that
$$({\mathcal D},{\mathcal K}_i,\psi): {\widetilde{X}}  \stretchx {\widetilde{X}}\,,
\quad\mbox{for } i =0,\dots, m-1.$$
Then the following conclusions hold:
\begin{itemize}
\item[$(b_1)\;$] The map $\psi$ induces chaotic dynamics on $m$ symbols in the set ${\mathcal D}$
relatively to
$({\mathcal K}_0,\dots,{\mathcal K}_{m-1});$
\item[$(b_2)\;$] For each sequence of $m$ symbols $\textbf{s}=(s_n)_n\in \{0,1,\dots,m-1\}^{\mathbb N},$
there exists a compact connected set ${\mathcal C}_{\textbf{s}}\subseteq {\mathcal K}_{s_0}$
which cuts the arcs between $X_{\ell}$ and $X_{r}$ in $X$ and such that,
for every $w\in {\mathcal C}_{\textbf{s}}\,,$
there is a sequence $(y_n)_{n}$ with $y_0 = w$ and
$$y_n \in {\mathcal K}_{s_n}\,,\quad \psi(y_n) = y_{n+1}\,,\;\forall\, n\geq 0.$$
The dimension of ${\mathcal C}_{\textbf{s}}$ at each point is at least $N-1.$
Moreover, $\mbox{\rm dim}({\mathcal C}_{\textbf{s}}\cap \vartheta X)\geq N - 2$ and
$\pi: h^{-1}({\mathcal C}_{\textbf{s}})\cap \partial I^N \to {\mathbb R}^{N-1}\setminus\{0\}$
is essential (where $\pi$ is defined as in \eqref{eq-3b.1} for $p_i=\tfrac 1 2,\,\forall\, i$).
\end{itemize}
\end{theorem}
\begin{proof}
The result easy follows by applying Theorem
\ref{th-3c.2} with the positions $X_k=X$ and $\psi_k=\psi,\,
\forall k\in \mathbb Z\,,$ and noting that, in view of Remark \ref{rem-3c.3},
conclusion $(b_2)$ is just a restatement of conclusion $(a_2)$ in Theorem \ref{th-3c.2},
while conclusion $(b_1)$ comes from conclusions $(a_1)$ and $(a_3)$ in Theorem \ref{th-3c.2}
and by Definition \ref{def-3c.3}.
\end{proof}

\bigskip

Several definitions of chaotic dynamics relate the behavior of the iterates of the map $\psi$
to a particular operator (the Bernoulli shift) acting on the set of sequences of $m$ symbols.
Our Definition \ref{def-3c.3} and the corresponding conclusion $(b_1)$ achieved in Theorem
\ref{th-3c.3} allow us to derive some facts in such a direction as well. To this aim, we first
recall some basic notions,
following \cite{Wi-88}.

\medskip

\noindent
Let $m\geq 2$ be a positive integer. We denote by
$$\Sigma_m = \{0,\dots, m-1\}^{\mathbb Z}$$
the set of the two$-$sided sequences of $m$ symbols. The set
$\Sigma_m$ can be endowed with a standard distance
\begin{equation}\label{eq-3c.dist}
d(\textbf{s}', \textbf{s}'') := \sum_{i\in {\mathbb Z}} \frac{|s'_i - s''_i|}{m^{|i| + 1}}\,,\quad
\mbox{ where }\; \textbf{s}'=(s'_i)_{i\in {\mathbb Z}}\,,\,
\textbf{s}''=(s''_i)_{i\in {\mathbb Z}}\,\in \Sigma_m\,,
\end{equation}
so that $(\Sigma_m,d)$ is a compact metric space.
The Bernoulli shift $\sigma$ is the homeomorphism on $\Sigma_m$ defined by
\begin{equation}\label{eq-3c.bs}
\sigma((s_i)_i):= (s_{i+1})_i\,
\end{equation}
and it represents one of the paradigms for chaotic dynamical systems in a (compact) metric space.
In particular (as shown in \cite[Theorem 7.12]{Wa-82}), $\sigma$ has positive topological entropy, expressed by
$$h_{\mbox{\footnotesize\rm top}}(\sigma) = \log(m)$$
(see \cite{Wa-82} for the pertinent definitions and more details).

Let $\Lambda$ be a compact metric space and let $\psi: \Lambda\to \Lambda$ be a continuous map.
We say that $\psi$ is
{\textit{semiconjugate to the two$-$sided $m-$shift}} if there exists a
continuous surjective mapping
$g: \Lambda \to \Sigma_m$ such that
\begin{equation}\label{eq-3c.con}
g \circ \psi = \sigma\circ g.
\end{equation}
In a similar manner, if we denote by
$$\Sigma_m^+ = \{0,\dots, m-1\}^{\mathbb N}$$
the set of the one$-$sided sequences of $m$ symbols, endowed with a distance
analogous to the one defined in \eqref{eq-3c.dist},
we say that $\psi$ is
{\textit{semiconjugate to the one$-$sided $m-$shift}} if there exists a
continuous surjective mapping
$g: \Lambda \to \Sigma_m^+$ such that
\eqref{eq-3c.con}
holds.

\bigskip

\bigskip

The following result (which is substantially a standard fact)
connects the concept of semiconjugation with the Bernoulli shift
to the one of chaotic dynamics expressed in Definition \ref{def-3c.3}.
Its proof could be
easily adapted from similar arguments previously appeared in the literature
(see, for instance
\cite{KeKoYo-01, KeYo-01}
for semidynamical systems induced by continuous maps of metric spaces),
but, for sake of completeness, we provide here all the details.

\begin{lemma}\label{lem-3c.1}
Let $Z$ be a metric space,  $\psi: Z\supseteq D_{\psi} \to Z$ be a map
which is continuous on a set ${\mathcal D}\subseteq D_{\psi}\,$
and induces therein chaotic dynamics on $m\geq 2$ symbols
(relatively to
$({\mathcal K}_0,\dots,{\mathcal K}_{m-1})$).
Then, there exists a nonempty compact set
$$\Lambda \subseteq \bigcup_{j=0}^{m-1} {\mathcal K}_j\,,$$
which is invariant for $\psi$ and such that $\psi|_{\Lambda}$ is
semiconjugate to the two$-$sided $m-$shift,
so that the topological entropy $h_{\mbox{\footnotesize\rm top}}(\psi)$ satisfies
$$h_{\mbox{\footnotesize\rm top}}(\psi)\geq \log(m).$$
Moreover, the subset
${\mathcal P}$ of $\Lambda$ made by the periodic points of $\psi$
is dense in $\Lambda$ and if we denote by $g: \Lambda\to \Sigma_m$
the continuous surjection in \eqref{eq-3c.con}, it holds also that
the counterimage through $g$ of any $k-$periodic sequence in
$\Sigma_m$ contains at least one $k-$periodic point.
\end{lemma}

\begin{proof}
Setting ${\mathcal K}:=\bigcup_{j=0}^{m-1} {\mathcal K}_j,$ we define
$$\Lambda_0:=\{w\in \mathcal K:\psi^i(w)\in\mathcal K,\,
\forall i\in \mathbb N\}=\bigcap_{i=0}^{+\infty} \psi^{-i}(\mathcal K)$$
and
$${\mathcal P}:= \{x\in \Lambda_0\,:\, \exists k\geq 1 \mbox{ with } \psi^{k}(x) = x\}.$$
Since $\mathcal K$ is compact and $\psi$ is continuous on ${\mathcal K},$
it follows immediately that also  $\Lambda_0$ is compact and that
$\psi(\Lambda_0)\subseteq \Lambda_0$
(that is, $\Lambda_0$ is invariant for $\psi$).
Let us now define $g_0:\Lambda_0\to
\Sigma_m^+,$ as
\begin{equation*}
g_0(w):=(s_i)_{i\in\mathbb N} \Leftrightarrow \psi^i(w)\in \mathcal K_{s_i},\,\forall i\in\mathbb N.
\end{equation*}
By Definition \ref{def-3c.3},
the map $g_0$ turns out to be
surjective and the counterimage through $g_0$ of any $k-$periodic
sequence in $\Sigma_m^+$ contains at least one $k-$periodic point
(belonging to ${\mathcal P}$).
The continuity of $g_0$ comes from the continuity of
$\psi$ on $\Lambda_0\,,$ the choice of the distance
$d$ in \eqref{eq-3c.dist} and the fact that the sets ${\mathcal K}_j$
are compact and pairwise disjoint. Actually, $g_0$ turns out to be uniformly
continuous as it is defined on a compact metric space.
A direct inspection shows that the relation in
\eqref{eq-3c.con} is satisfied and therefore the map $g_0$ induces a
semiconjugation between $\psi|_{\Lambda_0}$ and the
one$-$sided $m-$shift.

Let
$$\Sigma_m^{\,per}\subseteq \Sigma_m$$
be the set of the periodic two$-$sided sequences of $m$ symbols.
Since every two$-$sided periodic sequence of $m$ symbols determines
a one$-$sided periodic sequence of $m$ symbols (and viceversa),
we have that
the map $g_0|_{\mathcal P}$ may be considered as a function with values
in $\Sigma_m^{\,per}\,.$ In fact, for every $w\in {\mathcal P},$ we have
\begin{equation*}
g_0(w)=(s_i)_{i\in\mathbb Z}\in \Sigma_m^{\,per}
\Leftrightarrow \psi^i(w)\in \mathcal K_{s_i},\,\forall i\in\mathbb N.
\end{equation*}
Thus, we can define a uniformly continuous and surjective map
$$g_1: {\mathcal P}\to \Sigma_m^{\,per},$$
by setting, for each $w\in {\mathcal P}:$
\begin{equation}\label{eq-3c.gper}
g_1(w):=(s_i)_{i\in\mathbb Z}\in \Sigma_m^{\,per}
\Leftrightarrow \psi^i(w)\in \mathcal K_{s_i},\,\forall i\in\mathbb Z.
\end{equation}
Notice that
$$g_1\circ \psi (w) = \sigma \circ g_1 (w),\quad\forall\, w\in {\mathcal P},$$
where $\sigma$ is the two$-$sided Bernoulli shift on $m$ symbols defined in \eqref{eq-3c.bs}.

Now, setting
$$\Lambda:= {\overline{{\mathcal P}}}\subseteq \Lambda_0\,,$$
it holds that $\psi(\Lambda)\subseteq \Lambda$, so that $\Lambda$
is compact and invariant for $\psi.$ At last, we
extend the uniformly continuous surjective mapping
$$g_1 : {\mathcal P} \to \Sigma_m^{\,per} \subseteq \Sigma_m$$
to a continuous surjective function
$$g: \Lambda \to \Sigma_m\,,$$
such that
$$g\circ \psi (x) = \sigma \circ g (x),\quad\forall\, x\in \Lambda.$$
\noindent From the above proved semiconjugacy condition and by
\cite[Theorem 7.2]{Wa-82}
it follows that
$$h_{\mbox{\footnotesize\rm top}}(\psi)\geq h_{\mbox{\footnotesize\rm top}}(\sigma) = \log(m).$$
Hence we see that all the properties listed in the statement of the lemma are satisfied.
The proof is complete.
\end{proof}

Clearly, in view of the above lemma, conclusion $(b_1)$ in Theorem \ref{th-3c.3} can
be reformulated in terms of a semiconjugation between $\psi$ and a Bernoulli shift.

\bigskip
The next consequence of Theorem \ref{th-3c.2} deals with a situation
which occurs in some ODE models (see, e.g., \cite{DaPa-04, PaZa-00, PaZa-04a})
where there are generalized rectangles linked each other by a stretching map.
We confine ourselves to the simpler case in which only two objects are involved.
More general examples could be considered as well.

\begin{corollary}\label{cor-3c.1}
Let ${\widetilde{\mathcal A}_0}$ and ${\widetilde{\mathcal A}_1}$ be oriented $N-$dimensional rectangles
of a metric space $Z,$ with ${\mathcal A}_0\cap {\mathcal A}_1 = \emptyset,$ and let
$\psi: Z\supseteq D_{\psi}\to Z$ be a map.
Assume there exist compact sets
${\mathcal K}_{i,j}$ for $i, j \in \{0,1\},$
with
$${\mathcal K}_{i,j}\subseteq {\mathcal A}_i\,\cap D_{\psi}\,,\quad \forall\, i,j = 0,1$$
such that
$$({\mathcal K}_{i,j},\psi): {\widetilde{\mathcal A}_i}  \stretchx {\widetilde{\mathcal A}_j}\,,
\quad\forall\, i,j = 0,1.$$
Then the following conclusions hold:
\begin{itemize}
\item[$(c_1)\;$] For any two$-$sided sequence of two symbols $\textbf{s} = (s_k)_{k \in {\mathbb Z}}\in
\{0,1\}^{\mathbb Z},$
there exists a sequence $(w_k)_{k\in {\mathbb Z}}$
such that $w_k\in {\mathcal K}_{s_k ,s_{k+1}}\subseteq{\mathcal A}_{s_k}$
and $\psi(w_k) = w_{k+1}\,,$ for all ${k\in {\mathbb Z}}\,;$
\item[$(c_2)\;$] For any one$-$sided sequence of two symbols $\textbf{s} = (s_n)_{n \in {\mathbb N}}
\in \{0,1\}^{\mathbb N},$
there exists a compact connected set ${\mathcal C}_{\textbf{s}} \subseteq
{\mathcal K}_{s_0 ,s_{1}}\subseteq{\mathcal A}_{s_0}$
which cuts the arcs between ${\mathcal A}^{s_0}_{\ell}$ and ${\mathcal A}^{s_0}_{r}$ in
${\mathcal A}_{s_0}$ and such that, for every $w\in {\mathcal C}_{\textbf{s}}\,,$
there is a sequence $(y_n)_{n}$ with $y_0 = w$ and
$$y_n \in {\mathcal K}_{s_n,s_{n+1}}\,,\quad \psi(y_n) = y_{n+1}\,,\;\forall\, n\geq 0.$$
The dimension of ${\mathcal C}_{\textbf{s}}$ at each point is at least $N-1.$
Moreover, $\mbox{\rm dim}({\mathcal C}_{\textbf{s}}\cap \,\vartheta {\mathcal A}_{s_0})\geq N - 2$ and
$\pi: h^{-1}({\mathcal C}_{\textbf{s}})\cap \partial I^N \to {\mathbb R}^{N-1}\setminus\{0\}$
is essential (where $\pi$ is defined as in \eqref{eq-3b.1} for $p_i=\tfrac 1 2,\,\forall\, i$);
\item[$(c_3)\;$] For any two$-$sided sequence of two symbols $\textbf{s} =
(s_k)_{k \in {\mathbb Z}}\in \{0,1\}^{\mathbb Z}$
which is $m$-periodic ($m\geq 1$),
there exists a $m-$periodic sequence $(w_k)_{k\in {\mathbb Z}}$
such that $w_k\in {\mathcal K}_{s_k ,s_{k+1}}\subseteq{\mathcal A}_{s_k}$
and $\psi(w_k) = w_{k+1}\,,$ for all ${k\in {\mathbb Z}}\,.$
\end{itemize}
\end{corollary}
\begin{proof}
Recalling Remark \ref{rem-3c.3}, the result easy follows by
applying Theorem \ref{th-3c.2}  with the position $\psi_k=\psi,\,
\forall k\in \mathbb Z,$ and setting $X_k={\mathcal A}_0$ or
$X_k={\mathcal A}_1,$ according to the value of $s_k$ in the
considered sequence of two symbols.
\end{proof}

\bigskip

We end this paper with a result which applies Theorem \ref{th-3c.3}
to a framework which fits for possible applications
to the detection of chaos via computer assisted proofs
\footnote{$\,$ Indeed, it is not difficult to take advantage of the computations already performed
in articles like \cite{HuYa-05,YaLi-06,YaLiHu-05,YaYa-04}
(regarding topological horseshoes in the sense
of Kennedy and Yorke) and add to their conclusions
also the existence of infinitely many periodic solutions.}.
In view of it we preface the following definition
adapted from \cite{PaZa-04b, PaZa->, PiZa-05}.

\begin{definition}\label{def-3c.4} Let ${\widetilde{\mathcal M}}$ and ${\widetilde{\mathcal N}}$
be two oriented $N-$dimensional rectangles of the same metric space $Z.$ We say that
${\widetilde{\mathcal M}}$ is a \textit{vertical slab} of ${\widetilde{\mathcal N}}$
and write
$${\widetilde{\mathcal M}} \subseteq_{\,v} \, {\widetilde{\mathcal N}}$$
if ${{\mathcal M}}\subseteq {{\mathcal N}}$ and, either
$${{\mathcal M}}_{\ell} \subseteq {{\mathcal N}}_{\ell}\,\quad\mbox{and }\;
{{\mathcal M}}_{r} \subseteq {{\mathcal N}}_{r}\,,$$
or
$${{\mathcal M}}_{\ell} \subseteq {{\mathcal N}}_{r}\,\quad\mbox{and }\;
{{\mathcal M}}_{r} \subseteq {{\mathcal N}}_{\ell}\,,$$
so that any path in ${\mathcal M}$ joining the two sides of
${\mathcal M}^-$ is also a path in ${\mathcal N}$ and joins the two opposite
sides of ${\mathcal N}^-.$
\\
We say that
${\widetilde{\mathcal M}}$ is a \textit{horizontal slab} of ${\widetilde{\mathcal N}}$
and write
$${\widetilde{\mathcal M}} \subseteq_{\,h} \, {\widetilde{\mathcal N}}$$
if ${{\mathcal M}}\subseteq {{\mathcal N}}$ and
every path in ${\mathcal N}$ joining the two sides of
${\mathcal N}^-$ admits a sub-path in ${\mathcal M}$
that joins the two opposite
sides of ${\mathcal M}^-.$
\\
Given three oriented $N-$dimensional rectangles ${\widetilde{\mathcal A}},$
${\widetilde{\mathcal B}}$ and ${\widetilde{\mathcal E}}$ of the same metric space $Z,$
with ${\mathcal E}\subseteq {\mathcal A} \cap {\mathcal B},$
we say that ${\widetilde{\mathcal B}}$ \textit{crosses}
${\widetilde{\mathcal A}}$ \textit{in} ${\widetilde{\mathcal E}}$ and write
$${\widetilde{\mathcal E}}\in \{ {\widetilde{\mathcal A}}\pitchfork {\widetilde{\mathcal B}} \},$$
if
$${\widetilde{\mathcal E}}\subseteq_{\,v} \, {\widetilde{\mathcal A}}\quad\mbox{and } \;
{\widetilde{\mathcal E}}\subseteq_{\,h} \, {\widetilde{\mathcal B}}.$$
\end{definition}
\bigskip
The above definitions, which imitate the classical terminology in \cite[Ch.2.3]{Wi-88},
are topological in nature and therefore do not necessitate any metric assumption (like
smoothness, lipschitzeanity, or similar ones often required in the literature).
We also notice that the terms ``vertical'' and ``horizontal'' are employed
in a purely conventional manner; the vertical is the expansive
direction and the horizontal is the contractive one (in a quite broad sense).

\bigskip
Our next and final result (Theorem \ref{th-3c.cr})
depicts a situation when the domain and the codomain
of the mapping $\psi$ are two intersecting oriented
$N-$dimensional rectangles. A graphical
illustration of it can be found in Figure 3,
which is inspired by the Smale solenoid.

\begin{theorem}\label{th-3c.cr}
Let ${\widetilde{\mathcal A}}$ and ${\widetilde{\mathcal B}}$ be oriented $N-$dimensional rectangles
of a metric space $Z$ and let
${\mathcal D}\subseteq {\mathcal A}\cap D_{\psi}$ be a closed set such that
\begin{equation*}
({\mathcal D},\psi): {\widetilde{{\mathcal A}}}\stretchx {\widetilde{{\mathcal B}}}.
\end{equation*}
Assume there exist $m\geq 1$ oriented $N-$dimensional rectangles
$${\widetilde{\mathcal E}}_0\,,\dots,
{\widetilde{\mathcal E}}_{m-1} \,\in \{\,{\widetilde{\mathcal A}} \pitchfork
{\widetilde{\mathcal B}}\,\}.$$
Then, $\psi$ has at least a fixed point in each of the sets ${\mathcal D}\cap {\mathcal E}_i$
($i=0,\dots,m-1$). Moreover, if
$$m\geq 2 \quad\mbox{and } \;\; {\mathcal D}\cap {\mathcal E}_i\cap {\mathcal E}_j = \emptyset,\quad \mbox{for } i\not=j\;\,
(\,\forall \,i, j),$$
the following conclusions hold:
\begin{itemize}
\item[$(d_1)\;$] The map $\psi$ induces chaotic dynamics on $m$ symbols in the set ${\mathcal D}$
relatively to
$({\mathcal D} \cap {\mathcal E}_0,\dots,{\mathcal D} \cap {\mathcal E}_{m-1});$
\item[$(d_2)\;$] For each sequence of $m$ symbols $\textbf{s}=(s_n)_n\in \{0,1,\dots,m-1\}^{\mathbb N},$
there exists a compact connected set ${\mathcal C}_{\textbf{s}}\subseteq {\mathcal D} \cap {\mathcal E}_{s_0}$
which cuts the arcs between ${\mathcal E}_{\ell}^{s_0}$ and ${\mathcal E}_{r}^{s_0}$ in
${\mathcal E}_{s_0}$ and such that, for every $w\in {\mathcal C}_{\textbf{s}}\,,$
there is a sequence $(y_n)_{n}$ with $y_0 = w$ and
$$y_n \in {\mathcal E}_{s_n}\,,\quad \psi(y_n) = y_{n+1}\,,\;\forall\, n\geq 0.$$
The dimension of ${\mathcal C}_{\textbf{s}}$ at each point is at least $N-1.$
Moreover, $\mbox{\rm dim}({\mathcal C}_{\textbf{s}}\cap \vartheta {\mathcal A})\geq N - 2$ and
$\pi: h^{-1}({\mathcal C}_{\textbf{s}})\cap \partial I^N \to {\mathbb R}^{N-1}\setminus\{0\}$
is essential (where $\pi$ is defined as in \eqref{eq-3b.1} for $p_i=\tfrac 1 2,\,\forall\, i$).
\end{itemize}
\end{theorem}
\begin{proof}
First of all we show that
\begin{equation}\label{eq-3c.B}
(\mathcal D\cap{\mathcal E}_{i} ,\psi): {\widetilde{{\mathcal B}}}
\stretchx {\widetilde{{\mathcal B}}},\,\forall i=0,\dots,m-1.
\end{equation}
Indeed, let $\gamma$ be a path with ${\overline{\gamma}}\subseteq
\mathcal B$ and ${\overline{\gamma}}\cap \mathcal
B_{\ell}\not=\emptyset,\; {\overline{\gamma}}\cap \mathcal
B_{r}\not=\emptyset.$  Then, since ${\widetilde{\mathcal E_i}}
\subseteq_{\,h} \, {\widetilde{\mathcal B}},\,\forall
i=0,\dots,m-1,$ there exists a sub-path $\sigma (=\sigma_i)$ of
$\gamma$ such that ${\overline{\sigma}}\subseteq {\mathcal E_i}$ and
${\overline{\sigma}}\cap \mathcal E_{\ell}^i\not=\emptyset,\;
{\overline{\sigma}}\cap \mathcal E_{r}^i\not=\emptyset.$ Recalling now
that ${\widetilde{\mathcal E_i}} \subseteq_{\,v} \,
{\widetilde{\mathcal A}},\,\forall i=0,\dots,m-1,$ it holds that
${\overline{\sigma}}\subseteq {\mathcal E}_i \subseteq {\mathcal A}$
and ${\overline{\sigma}}\cap \mathcal A_{\ell}\not=\emptyset,\;
{\overline{\sigma}}\cap \mathcal A_{r}\not=\emptyset.$ Finally, since
$({\mathcal D},\psi): {\widetilde{\mathcal A}}\stretchx
{\widetilde{\mathcal B}},$ there is a sub-path $\eta (=\eta_i)$ of
$\sigma$ such that $\overline{\eta}\subseteq \mathcal D\cap{\mathcal
E}_{i} ,\,\psi({\overline{\eta}})\subseteq \mathcal B,$ with
$\psi({\overline{\eta}})\cap \mathcal B_{\ell}\not=\emptyset,\;
\psi({\overline{\eta}})\cap \mathcal B_{r}\not=\emptyset.$ In this way
we have proved that any path $\gamma$ with
${\overline{\gamma}}\subseteq \mathcal B$ and ${\overline{\gamma}}\cap
\mathcal B_{\ell}\not=\emptyset,\; {\overline{\gamma}}\cap \mathcal
B_{r}\not=\emptyset$ admits a sub-path $\eta$ such that
$\overline{\eta}\subseteq {\mathcal D}\cap{\mathcal E}_i$ and
$\psi(\overline{\eta})\subseteq {\mathcal B}$ with
$\psi(\overline{\eta})\cap
\mathcal B_{\ell}\not=\emptyset,\; \psi(\overline{\eta})\cap \mathcal
B_{r}\not=\emptyset.$
Therefore
the condition in \eqref{eq-3c.B} has been checked
and the
existence of at least a fixed point for $\psi$ in ${\mathcal
D\cap{\mathcal E}}_{i}\,$ ($\forall i=0,\dots,m-1$) follows by
Theorem \ref{th-3c.1}. To prove the remaining part of the
statement, we apply Theorem \ref{th-3c.3} with the positions
$X=\mathcal B$ and $\mathcal K_i={\mathcal D\cap{\mathcal
E}}_{i},\,$ for $i=0,\dots,m-1.$
\end{proof}

\newpage

\includegraphics[scale=0.3]{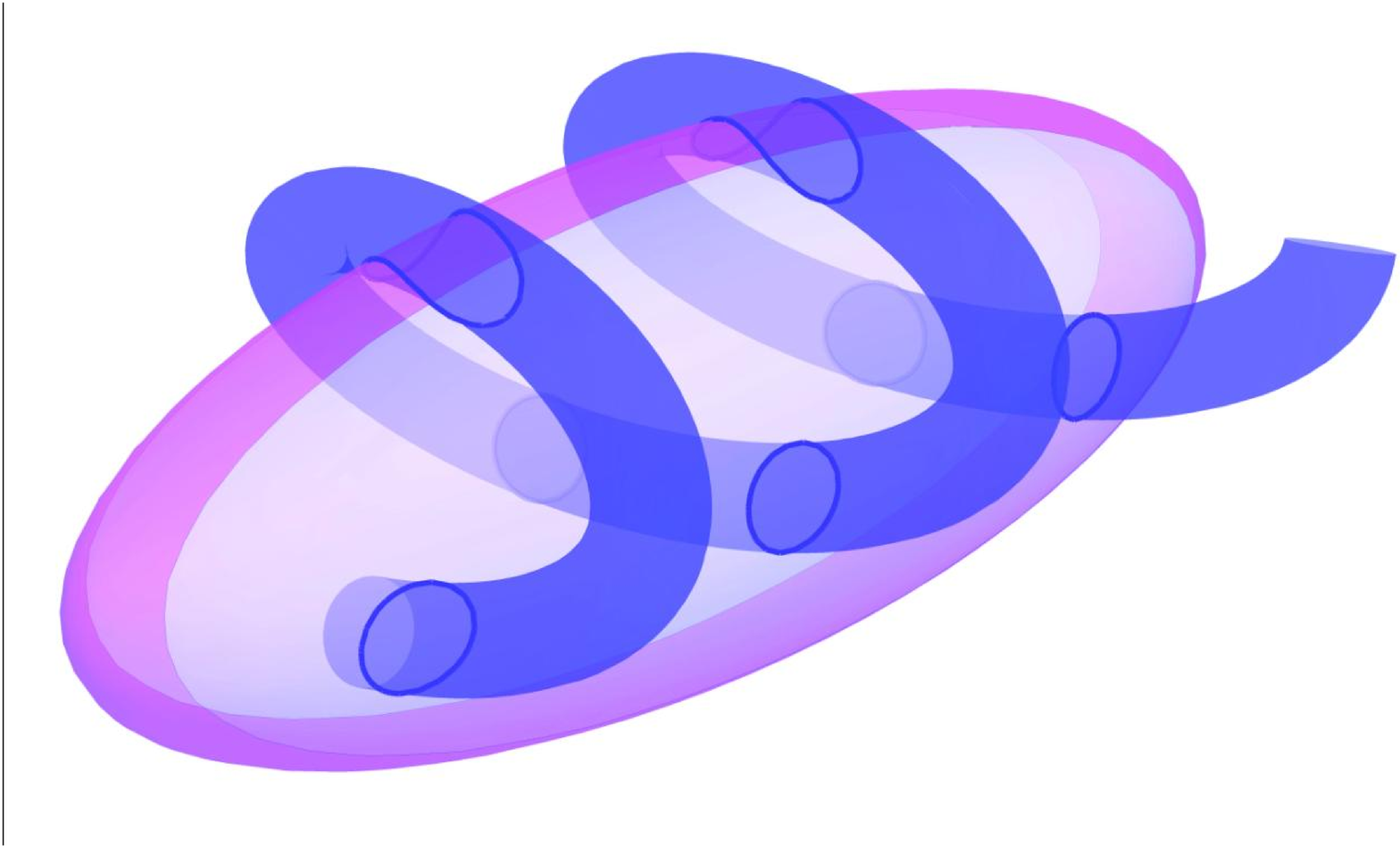}
\begin{quote}
{\small{Figure 3. The ellipsoidal body ${\mathcal A}$ is stretched by a continuous mapping $\psi$
to a subset of the spiral$-$like set ${\mathcal B}.$ Both ${\mathcal A}$ and ${\mathcal B}$ are
3$-$dimensional generalized rectangles that we orientate as follows: the compact sets
${\mathcal A}_{\ell}$ and ${\mathcal A}_{r}$ are the
closure of the two components of $\vartheta{\mathcal A}= \partial{\mathcal A}$ which
are obtained after removing the darker part of the ``lateral'' surface;
the compact sets ${\mathcal B}_{\ell}$ and ${\mathcal B}_{r}$ are the two discs at the ends of the spiral body
${\mathcal B}$ (the order in which we label the two parts of ${\mathcal A}^-$ and
${\mathcal B}^-$ can be chosen arbitrarily). According to Remark \ref{rem-3c.2},
the stretching condition $\psi: {\widetilde{\mathcal A}}\stretchx {\widetilde{\mathcal B}}$
is fulfilled if we assume that
$\psi({\mathcal A})\subseteq {\mathcal B}$ and that
$\psi({\mathcal A}_{\ell})\subseteq {\mathcal B}_{\ell}\,,$
as well as $\psi({\mathcal A}_{r})\subseteq {\mathcal B}_{r}\,.$
Note that we do not require $\psi$ to be a homeomorphism, nor
$\psi({\mathcal A})= {\mathcal B}.$ It is not even necessary that the end sets
${\mathcal B}_{\ell}$ and ${\mathcal B}_{r}$ of ${\mathcal B}$
lie outside ${\mathcal A}.$
Among the five intersections between
${\mathcal A}$ and ${\mathcal B},$ only two (namely, the ones visible as a full crossing of
the spiral$-$like set across the ellipsoidal body, that we call ${\mathcal E}_0$ and
${\mathcal E}_1$) correspond to a crossing in the sense of
Definition \ref{def-3c.4}. Therefore, Theorem \ref{th-3c.cr}
ensures the existence of at least a fixed point for $\psi$ both in ${\mathcal E}_0$ and ${\mathcal E}_1$
and, moreover, $\psi$ induces chaotic dynamics on two symbols in ${\mathcal A}$
(relatively to ${\mathcal E}_0$ and ${\mathcal E}_1$).
Even if the drawn figures look smooth,
there is no need of any regularity assumption neither for the sets
(except of being homeomorphic to a cube) nor for their intersections.
}
}
\end{quote}

\bigskip

\noindent
To conclude, we stress the fact that the definition of oriented $N-$dimensional rectangle can be slightly modified in
order to take into account suitable perturbations of the domain and of the
map. A similar setting has been analyzed, for instance, in \cite{CaZg-06},
\cite{KeYo-98}, \cite{WoZg-05} and \cite{Zg-99}.
A development of these topics (which have a relevant interest from the point
of view of the applications) will be studied in a
subsequent work, using our approach.

\bigskip
\noindent
\textbf{Acknowledgment}.
The authors thank professor Gianluca Gorni for his help in the use of Mathematica software.

\bigskip

\bigskip

\noindent
\address{
$~$
\\
{
\textsc{Fabio Zanolin}\\
University of Udine, Department of Mathematics and Computer Science,\\
via delle Scienze 206, 33100 Udine, Italy.
\\
mailto: \itshape{zanolin@dimi.uniud.it}}
\\
$~$
\\
{
\textsc{Marina Pireddu}\\
University of Udine, Department of Mathematics and Computer Science,\\
via delle Scienze 206, 33100 Udine, Italy.
\\
mailto: \itshape{pireddu@dimi.uniud.it}}
}

\end{document}